\documentclass[11pt,reqno]{amsart}

\usepackage{amsfonts,amsmath,amsthm,amssymb,mathtools}
\usepackage{enumitem}
\usepackage{bm}
\usepackage[letterpaper,margin=2.5cm]{geometry}
\usepackage[hidelinks]{hyperref}

\newtheorem{theorem}{Theorem}[section]
\newtheorem{lemma}[theorem]{Lemma}
\newtheorem{proposition}[theorem]{Proposition}

\theoremstyle{definition}
\newtheorem{definition}[theorem]{Definition}

\newtheorem{problem}[theorem]{Problem}

\newtheorem*{ack}{Acknowledgments}

\theoremstyle{remark}
\newtheorem{remark}[theorem]{Remark}

\numberwithin{equation}{section}

\DeclarePairedDelimiter{\parens}{(}{)}

\newcommand{\inn}[2]{\langle #1,#2\rangle}

\newcounter{rom}
\renewcommand{\therom}{(\roman{rom})}
{\end{list}}

\title[Stable spacelike capillary hypersurfaces]
      {Rigidity of stable spacelike capillary hypersurfaces in de~Sitter and Minkowski spaces}

\author{Hui Ma}
\address{Department of Mathematical Sciences, Tsinghua University,
Beijing 100084, P.R. China} 
\email{ma-h@tsinghua.edu.cn}
\author{Jiaxu Ma}
\address{Department of Mathematical Sciences, Tsinghua University,
Beijing 100084, P.R. China}
\email{mjx22@mails.tsinghua.edu.cn}
\author{Mingxuan Yang}
\address{Academy of Mathematics and Systems Science, the Chinese Academy of Sciences, Beijing 100190, China}
\email{ymx20@amss.ac.cn}


\date{}

\begin{document}

\begin{abstract}
We prove a rigidity theorem for compact spacelike capillary hypersurfaces
in de~Sitter and Minkowski spaces: volume-preserving stability forces total
umbilicity when the support is a spacelike totally umbilical hypersurface
of nonnegative intrinsic curvature. 
Using the light-cone model, 
we construct conformal Killing fields tangent to the support and derive a unified 
Minkowski-type formula valid in all Lorentzian space forms. The resulting
mean-zero functions satisfy an inhomogeneous Jacobi equation and
the linearized capillary Robin boundary condition,
and form canonical finite-dimensional test families.
A finite-trace identity, supplemented in the de~Sitter cases by a
nonpositive Dirichlet Green correction, detects the umbilicity defect
\(n|h|^2-H^2\) with a definite sign; hence, every
non-totally-umbilical hypersurface admits an admissible test function with
positive second variation. The construction extends to supports of negative
intrinsic curvature, where a unique timelike parameter direction prevents
the finite trace from being sign-definite.
\end{abstract}

\maketitle

\section{Introduction}
\label{sec:intro}

The stability of constant-mean-curvature (CMC) hypersurfaces is a
fundamental variational problem in submanifold geometry. In Riemannian
space forms, the classical theorem of Barbosa, do Carmo, and
Eschenburg~\cite{BCE} asserts that every closed stable CMC hypersurface
is totally umbilical. 
For spacelike hypersurfaces in Lorentzian space forms, the timelike
character of the unit normal changes the sign structure of the
second-variation formula relative to the Riemannian case. Stability is
therefore naturally formulated as an infinitesimal area-maximizing property.
The variational theory of
spacelike CMC hypersurfaces has been developed in several important
works; see, among others, Cheng--Yau~\cite{CY},
Treibergs~\cite{T82}, Akutagawa~\cite{Akutagawa87},
Montiel~\cite{Montiel88}, Oliker~\cite{Oliker92}, and
Barbosa--Oliker~\cite{BO93}.

The capillary problem is the natural boundary analogue of the CMC problem.
In the Riemannian setting, its stability and rigidity theory originated
with the work of Ros and Vergasta~\cite{RosVergasta95} on free-boundary
hypersurfaces in balls and Ros and Souam~\cite{RosSouam97} on capillary
hypersurfaces in balls. Subsequent developments treated wedges,
half-spaces, slabs, planar supports, Euclidean balls, and
horospherical supports in hyperbolic space; see
\cite{CK15,AinouzSouam16,LiXiong17,LiXiong18,GuoWangXia22}.
A general space-form rigidity theorem was proved by Wang and
Xia~\cite{WangXia19}: every immersed stable capillary hypersurface in a
geodesic ball of a Riemannian space form is totally umbilical. Their
argument combines suitable Minkowski-type formulas with the
volume-preserving stability inequality.

We formulate the corresponding problem in a time-oriented Lorentzian
manifold \((\overline M^{n+1},\overline g)\). Let
\(\mathcal P\subset\overline M\) be a spacelike support hypersurface and let
\[
x\colon\Sigma^n\longrightarrow\overline M
\]
be a compact spacelike immersion satisfying
\(x(\partial\Sigma)\subset\mathcal P\). Choose unit timelike normals
\(\nu\) and \(N\) to \(\Sigma\) and \(\mathcal P\), respectively, in the
same time cone. The contact angle, or relative rapidity, is
the positive function \(\gamma\) determined by
\[
\overline g(\nu,N)=-\cosh\gamma
\qquad\text{on }\partial\Sigma.
\]
The hypersurface \(\Sigma\) is capillary if both its mean curvature and
its contact angle are constant. Equivalently, it is a critical
point, among volume-preserving admissible variations, of the capillary
functional
\[
\mathcal E(t)=\mathcal A(t)-\cosh\gamma_0\,\mathcal W(t),
\]
where \(\mathcal A(t)\) is the area of the varied hypersurface,
\(\mathcal W(t)\) is the signed wetting energy associated with the support,
and \(\gamma_0>0\) is constant; see
\cite{ConcusFinn,Finn} for the classical capillary framework.

The Lorentzian stability problem is structurally different from its
Riemannian counterpart.
Previous works~\cite{Lopez06,Lopez08,PyoSeo11,ChenDengXieYin25} address symmetry and classification
questions for particular Lorentzian hypersurfaces with constant-angle boundary conditions, but do not yield
a volume-preserving stability-rigidity theorem in higher dimensions. To
the best of our knowledge, no previous theorem of this generality was available for compact spacelike capillary hypersurfaces.

The purpose of this paper is to prove a rigidity theorem for stable spacelike capillary hypersurfaces with
spacelike totally umbilical supports of nonnegative intrinsic curvature.
The key point is the construction of a canonical finite-dimensional
module of admissible Jacobi--Robin test functions. Rather than estimating
the second variation of a single test function, we compute an invariant
trace over this module. For the support geometries occurring in the main
theorem, the trace has a definite sign and recovers a weighted integral of
the umbilicity defect
\[
n|h|^2-H^2.
\]
This turns the stability condition directly into a rigidity conclusion.

We denote by \(\mathbb Q^{n+1}_1(c)\) the \((n+1)\)-dimensional
Lorentzian space form of constant sectional curvature \(c\). In particular,
\(\mathbb Q^{n+1}_1(1)\) is de~Sitter space and
\(\mathbb Q^{n+1}_1(0)\) is Minkowski space.

\begin{theorem}
\label{thm:intro-criterion}
Let \(\mathcal P\subset\mathbb Q^{n+1}_1(c)\), \(c\in\{0,1\}\), be a
connected spacelike totally umbilical hypersurface with nonnegative intrinsic
sectional curvature.
Let \(\Sigma^n\), \(n\geq2\), be a compact connected spacelike
capillary hypersurface with nonempty boundary supported on
\(\mathcal P\). Denote its constant contact angle by
\(\gamma>0\).
If \(\Sigma\) is stable under volume-preserving admissible variations,
then \(\Sigma\) is totally umbilical.
\end{theorem}

By the classification of spacelike totally umbilical hypersurfaces,
the support geometries in Theorem~\ref{thm:intro-criterion} are precisely:

\begin{itemize}
\item a spacelike totally geodesic hyperplane in Minkowski space;
\item a round spacelike totally umbilical hypersurface in de~Sitter space;
\item a horospherical spacelike totally umbilical hypersurface in de~Sitter
space.
\end{itemize}

The proof has three main ingredients. First, the light-cone model produces
support-tangent conformal Killing fields for spacelike totally umbilical supports in arbitrary Lorentzian space forms. The associated
divergence identities yield a unified Minkowski-type formula and, in particular,
canonical mean-zero functions. This part of the construction does not
require the support to have nonnegative intrinsic curvature.

Second, the resulting functions satisfy an inhomogeneous Jacobi
equation and the linearized capillary Robin boundary condition. They form a finite-dimensional module of admissible test functions
for the volume-preserving index form. This Jacobi--Robin structure is
intrinsic to the capillary problem.

Third, we compute the invariant trace of the index form over the canonical
module. For the three support geometries listed above, a nonpositive
Dirichlet Green correction in the de~Sitter cases reduces the trace
estimate to a positive weighted integral of \(n|h|^2-H^2\). If the umbilicity defect does not vanish identically,
at least one admissible test function has
positive second variation, contradicting stability.

The curvature assumption in Theorem~\ref{thm:intro-criterion} enters only
at this final finite-trace step. The light-cone construction, the
Minkowski-type formula, and the Jacobi--Robin identities remain valid when the
support has negative intrinsic curvature. In that case, however, the
parameter module has Lorentzian signature. The invariant trace is therefore no longer
sign-definite. This identifies a precise obstruction to extending the
present rigidity argument to negative-curvature supports.

The rest of the paper is organized as follows.
Section~\ref{sec:prelim} introduces the boundary geometry and the
variational formulas. Sections~\ref{sec:light_cone} and
\ref{sec:Jacobi-Robin} develop the light-cone construction, the 
Minkowski-type formula, and the Jacobi--Robin identities.
Sections~\ref{sec:finite-trace} and \ref{sec:rigidity} establish the
finite-trace identity and prove the rigidity theorem.
Section~\ref{sec:negative-curvature} analyzes the signature obstruction
for negative-curvature supports, while Appendix~A gives a self-contained
derivation of the Lorentzian capillary second-variation formula.

\section{Geometric setup and variational formulas}
\label{sec:prelim}

\subsection{Capillary hypersurfaces and boundary geometry}

Let \((\overline M,\overline g)\) be an oriented and time-oriented
\((n+1)\)-dimensional Lorentzian spacetime, and let
\(\mathcal P\subset \overline M\) be a spacelike hypersurface with a chosen
future-directed unit timelike normal \(N\). Let \(x:\Sigma\to \overline M\) be a
spacelike immersion of a compact \(n\)-manifold with boundary such that
\(x(\partial\Sigma)\subset\mathcal P\).

The time orientation determines a unique future-directed unit timelike normal
field \(\nu\) on all of \(\Sigma\). We denote by
\(g:=x^*\overline g\)
the induced Riemannian metric on \(\Sigma\), and define the second fundamental
form of \(\Sigma\) with respect to \(\nu\) by
\[
h(X,Y):=\overline g(\overline\nabla_X\nu,Y),
\qquad X,Y\in T\Sigma.
\]
The mean curvature is defined by
\[
H:=\operatorname{tr}_g h
=\sum_{i=1}^n h(e_i,e_i)
=\operatorname{div}_\Sigma\nu,
\]
where \(\{e_i\}_{i=1}^n\) is any local \(g\)-orthonormal frame on \(\Sigma\).

We denote by \(h^{\mathcal P}\) the second fundamental form of
\(\mathcal P\) with respect to \(N\), defined by
\[
h^{\mathcal P}(X,Y)
:=
\overline g(\overline\nabla_XN,Y),
\qquad X,Y\in T\mathcal P.
\]

Along \(\partial\Sigma\), the fields \(\nu\) and \(N\) satisfy
\[
\overline g(\nu,N)=-\cosh\gamma.
\]
Throughout, we assume that the intersection is transverse, so that
\(\gamma>0\).

Let \(\mu\) be the outward unit conormal of \(\partial\Sigma\) in \(\Sigma\).
We fix the side of the support so that the boundary frame satisfies
\begin{equation}\label{eq:boundary-frame}
N=\cosh\gamma\,\nu-\sinh\gamma\,\mu.
\end{equation}

Define
\[
    \overline\nu=-\sinh\gamma\,\nu+\cosh\gamma\,\mu.
\]
Then \(\overline\nu\) is the unit conormal of \(\partial\Sigma\) in
\(\mathcal P\) compatible with the chosen side.
Note that the normal space of \(\partial\Sigma\) in \(\overline M\) is spanned
by either of the ordered frames \((\nu,\mu)\) and
\((N,\overline\nu)\), which induce the same orientation.
We also have
\begin{equation}\label{eq:boundary-frame-inverse}
\begin{aligned}
   \nu&=\cosh\gamma\,N+\sinh\gamma\,\overline\nu,\\
   \mu&=\sinh\gamma\,N+\cosh\gamma\,\overline\nu.
\end{aligned}
\end{equation}

\subsection{The capillary energy functional and its variation}

Let \(x:(-\varepsilon,\varepsilon)\times\Sigma\to \overline M\) be an admissible variation 
of the immersion \(x_0:=x(0,\cdot)\); that is,
\(x(t,\partial\Sigma)\subset\mathcal{P}\)
for every \(t\in(-\varepsilon,\varepsilon)\).

The volume, area and the wetting energy functional are defined by  
\[
\mathcal V(t)=\int_{[0,t]\times\Sigma}x^*dV_{\overline M}, \quad
\mathcal A(t)=\int_{\Sigma_t} d\Sigma_t, \quad
\mathcal W(t)=\int_{[0,t]\times\partial\Sigma }x^*dV_{\mathcal{P}},
\]
where \(d\Sigma_t\) is the area element of
\(\Sigma_t:=x_t(\Sigma)\), and \(dV_{\mathcal P}\) is the
\(n\)-dimensional volume form of \(\mathcal P\). Let
\[
Y_t=\frac{\partial x_t}{\partial t},
\qquad
f_t=-\overline g(Y_t,\nu_t)
\]
denote the variational field and its normal speed, respectively.
Fix a constant contact angle \(\gamma_0>0\). The capillary energy functional is
\[
\mathcal E(t)=\mathcal A(t)-\cosh\gamma_0\,\mathcal W(t).
\]

\begin{lemma}[First variation of the capillary energy]\label{lem:first-var-unified}
  Let \(x_t\colon\Sigma\to \overline M\) be an admissible variation of
  \(x_0=x\).
  Then
  \begin{equation}\label{eq:first-var-cap-unified}
    \mathcal E'(t) = \int_{\Sigma_t} H(t)f_t\,d\Sigma_t
    + \int_{\partial\Sigma_t} \overline g\bigl(Y_t,\,\mu_t - \cosh\gamma_0\,\overline\nu_t\bigr)\,dS_t .
  \end{equation}
\end{lemma}

\begin{proof}
  The area functional has the standard first variation
  \[
    \mathcal A'(t) = \int_{\Sigma_t} \operatorname{div}_{\Sigma_t} Y_t\,d\Sigma_t .
  \]
  Decompose \(Y_t = f_t\nu_t + Y_t^\top\).
  Since \(\nu_t\) is a unit timelike normal,
  \(\operatorname{div}_{\Sigma_t}\nu_t=H(t)\). The divergence theorem
  therefore gives
  \[
    \mathcal A'(t) = \int_{\Sigma_t} H(t)f_t\,d\Sigma_t + \int_{\partial\Sigma_t} \overline g(Y_t,\mu_t)\,dS_t .
  \]

  The quantity \(\mathcal W(t)\) is the signed wetting energy associated
  with the region traced out by the moving contact boundary on \(\mathcal P\).
  Since \(x_t(\partial\Sigma)\subset\mathcal P\), the variation field
  \(Y_t\) is tangent to \(\mathcal P\) along \(\partial\Sigma_t\). Hence,
  \[
    \mathcal W'(t) = \int_{\partial\Sigma_t} \overline g(Y_t,\overline\nu_t)\,dS_t ,
  \]
  where \(\overline\nu_t\) is the corresponding unit conormal of
  \(\partial\Sigma_t\) in \(\mathcal P\).

  Combining these identities gives \eqref{eq:first-var-cap-unified}.
\end{proof}

The first variation of the volume is
\[
    \mathcal V'(t)=\int_{\Sigma_t}f_t\,d\Sigma_t.
\]
Accordingly, an admissible variation is volume-preserving if
\[
\int_{\Sigma_t}f_t\,d\Sigma_t=0
\]
for every \(t\).

\begin{definition}
A spacelike immersion \(x\colon\Sigma\to \overline M\) satisfying
\(x(\partial\Sigma)\subset\mathcal P\) is called capillary if it is a
critical point of \(\mathcal E\) under admissible volume-preserving
variations.
\end{definition}

\begin{proposition}
The immersion \(x\colon\Sigma\to \overline M\) is capillary if and only if \(H\)
is constant and
\[
\overline g(\nu,N)=-\cosh\gamma_0
\qquad\text{along }\partial\Sigma.
\]
\end{proposition}

Henceforth, for a capillary hypersurface, we write \(\gamma=\gamma_0\)
for its constant contact angle. From now on, we assume that the support \(\mathcal P\) is totally umbilical.

\begin{lemma}[Second variation formula]\label{lem:second-variation}
Let \(\Sigma\) be a spacelike capillary hypersurface supported on \(\mathcal P\). For every admissible volume-preserving variation with initial normal speed
\(f=-\overline g(Y,\nu)\),
    \begin{equation}\label{eq:second-variation}
        \mathcal E''(0)=Q(f,f):=\int_\Sigma fJf\,d\Sigma
        -\int_{\partial\Sigma}f(\nabla_\mu f-qf)\,dS,
    \end{equation}
where
\[
q=-\frac{1}{\sinh\gamma}
  h^{\mathcal P}(\overline\nu,\overline\nu)
  +\coth\gamma\,h(\mu,\mu),
\]
and
\[
J=\Delta-|h|^2-\overline{\operatorname{Ric}}(\nu,\nu)
\]
is the Jacobi operator.

A spacelike capillary hypersurface \(\Sigma\) is stable if and only if
\[
Q(f,f)\leq0
\]
for every \(f\in C^\infty(\Sigma)\) satisfying
\[
\int_\Sigma f\,d\Sigma=0.
\]
\end{lemma}
A proof with the Lorentzian sign conventions used in this paper is given in Appendix~\ref{sec:second-var-unified}.

\begin{remark}
The equivalence in Lemma~\ref{lem:second-variation} uses the standard
realization lemma. If
\(f\in C^\infty(\Sigma)\) satisfies
\(\int_\Sigma f\,d\Sigma=0,\)
then \(f\) is the initial normal speed of an admissible volume-preserving
variation.  Along the boundary, one first prescribes
\[
Y=f\nu-f\coth\gamma\,\mu,
\]
which is tangent to \(\mathcal P\) because
\[
N=\cosh\gamma\,\nu-\sinh\gamma\,\mu.
\]
A two-parameter correction, together with the implicit function theorem,
then yields exact volume preservation without changing the initial
normal speed.
\end{remark}

\begin{proposition}\label{prop:mu-principal}
  Let \(\Sigma\) be a spacelike capillary hypersurface in a Lorentzian space form \(\overline M^{n+1}(c)\), supported on a totally umbilical spacelike hypersurface \(\mathcal P\) and meeting \(\mathcal P\) at a constant contact angle \(\gamma>0\).
  Then \(\mu\) is a principal direction of \(\Sigma\) along \(\partial\Sigma\); that is,
  \[
    h(e,\mu)=0\qquad\text{for all } e\in T(\partial\Sigma).
  \]
  In particular,
  \begin{equation}\label{eq:Dnu}
    \overline\nabla_\mu\nu = h(\mu,\mu)\,\mu .
  \end{equation}
\end{proposition}

\begin{proof}
  For \(e\in T(\partial\Sigma)\), the total umbilicity of
  \(\mathcal P\) gives
  \[
  \begin{aligned}
    h(e,\mu)
    &= \overline{g}\bigl(\overline{\nabla}_e\nu,\mu\bigr) \\
    &= \overline{g}\bigl(\overline{\nabla}_e(\cosh\gamma N+\sinh\gamma\overline\nu),\;
                \sinh\gamma N+\cosh\gamma\overline\nu\bigr) \\
    &= \overline{g}\bigl(\overline{\nabla}_e N,\overline\nu\bigr)
       = h^{\mathcal{P}}(e,\overline\nu)=0.
  \end{aligned}
  \qedhere
  \]
\end{proof}

\begin{proposition}\label{prop:boundary-relations}
Let \(x:\Sigma^n\to \overline M^{n+1}\) be a spacelike immersion whose boundary
lies on a spacelike totally umbilical support \(\mathcal P\). Assume that
\[
h^{\mathcal P}=\kappa\,\overline g|_{T\mathcal P},
\qquad
h^{\mathcal P}(X,Y)=\overline g(\overline\nabla_XN,Y),
\]
and that along \(\partial\Sigma\)
\[
N=\cosh\gamma\,\nu-\sinh\gamma\,\mu,
\qquad
\gamma>0.
\]
Let \(\{e_\alpha\}_{\alpha=1}^{n-1}\) be a local orthonormal frame on
\(\partial\Sigma\). Define the unnormalized mean curvature of
\(\partial\Sigma\subset\mathcal P\) by
\[
\widehat H:=\sum_{\alpha=1}^{n-1}
\overline g(\overline\nabla_{e_\alpha}\overline\nu,e_\alpha),
\]
and define the second fundamental form of \(\partial\Sigma\subset\Sigma\) by
\[
\widetilde h(e_\alpha,e_\beta)
:=\overline g(\overline\nabla_{e_\alpha}\mu,e_\beta).
\]
Then the following identities hold along \(\partial\Sigma\):
  \begin{align}
    h(\mu,\mu) &= H - (n-1)\kappa\cosh\gamma - \sinh\gamma\,\widehat H, \label{eq:hmm} \\
    \widetilde h(e_\alpha,e_\beta) &= -\frac{1}{\sinh\gamma} h^{\mathcal{P}}(e_\alpha,e_\beta)
                                 + \coth\gamma\,h(e_\alpha,e_\beta) .
                                 \label{eq:tildeh}
  \end{align}
\end{proposition}

\begin{proof}
  By definition,
  \(H = h(\mu,\mu) + \sum_{\alpha=1}^{n-1} h(e_\alpha,e_\alpha)\).
  Using \(\nu=\cosh\gamma\,N+\sinh\gamma\,\overline\nu\) and
  \(\overline g(N,e_\alpha)=\overline g(\overline\nu,e_\alpha)=0\), we compute
  \begin{align*}
    \sum_{\alpha=1}^{n-1} h(e_\alpha,e_\alpha)
    &= \sum_{\alpha=1}^{n-1} \overline g(\overline\nabla_{e_\alpha}\nu, e_\alpha) \\
    &= \cosh\gamma \sum_{\alpha=1}^{n-1} \overline g(\overline\nabla_{e_\alpha} N, e_\alpha)
       + \sinh\gamma \sum_{\alpha=1}^{n-1} \overline g(\overline\nabla_{e_\alpha} \overline\nu, e_\alpha) \\
    &= \cosh\gamma \sum_{\alpha=1}^{n-1} h^{\mathcal P}(e_\alpha,e_\alpha) + \sinh\gamma\,\widehat H.
  \end{align*}
 Therefore,
  \[
    H = h(\mu,\mu) + (n-1)\kappa\cosh\gamma + \sinh\gamma\,\widehat{H},
  \]
  which is equivalent to \eqref{eq:hmm}.

  For the second identity, decompose \(\mu = \sinh\gamma N + \cosh\gamma \overline\nu\) and
  differentiate along \(e_\alpha\):
  \begin{align*}
    \widetilde h(e_\alpha,e_\beta)= \overline g(\overline\nabla_{e_\alpha}\mu, e_\beta) = \sinh\gamma\,h^{\mathcal{P}}(e_\alpha,e_\beta)
       + \cosh\gamma\,\overline g(\overline\nabla_{e_\alpha}\overline\nu, e_\beta).
  \end{align*}
  Now use the inverse relation \(\overline\nu = \cosh\gamma\mu - \sinh\gamma\nu\) to replace the
  last term:
  \begin{align*}
    \overline g(\overline\nabla_{e_\alpha}\overline\nu, e_\beta)
    &= \cosh\gamma\,\widetilde h(e_\alpha,e_\beta) - \sinh\gamma\,h(e_\alpha,e_\beta).
  \end{align*}
Therefore,
  \[
    \widetilde h = -\frac{1}{\sinh\gamma}h^{\mathcal{P}} + \coth\gamma\,h,
  \]
  which proves \eqref{eq:tildeh}.
\end{proof}

\section{The light-cone model and the unified Minkowski-type formula}
\label{sec:light_cone}

\subsection{The light-cone model and conformal Killing fields}
\label{subsec:light-cone}

Let
\[
(\mathbb V,\langle\cdot,\cdot\rangle_{\mathbb V})
=\mathbb R^{n+1,2}
\]
be the pseudo-Euclidean space of signature \((n+1,2)\).
In coordinates,
\[
    \langle \xi,\xi\rangle_{\mathbb V}
    =-(\xi^0)^2+\sum_{i=1}^{n+1}(\xi^i)^2-(\xi^{n+2})^2,
\]
for any \(\xi=(\xi^0,\xi^1,\cdots,\xi^{n+2})\in \mathbb{V}\).
Fix \(c\in\{1,0,-1\}\) and choose a constant vector \(E_c\in\mathbb V\) satisfying
\[
    \langle E_c,E_c\rangle_{\mathbb V}=-c.
\]
Set
\[
    \mathbb Q^{n+1}_1(c)
    =\left\{\xi\in\mathbb V\mid \langle\xi,\xi\rangle_{\mathbb V}=0,\,
   \langle \xi,E_c\rangle_{\mathbb V}=1\right\}.
\]
Equipped with the induced metric \(\overline g\), this submanifold is a
Lorentzian space form of constant sectional curvature \(c\). Indeed, let
\(D\) denote the flat connection of \(\mathbb V\) and
\(\overline\nabla\) the Levi-Civita connection of
\((\mathbb Q^{n+1}_1(c),\overline g)\). Then, for
\(U,V\in T\mathbb Q^{n+1}_1(c)\),
\[
    D_UV
    =\overline\nabla_UV-\overline g(U,V)(E_c+c\xi).
\]
Hence, the Gauss equation gives
\[
    \overline R(U,V)W
    =c\bigl(\overline g(V,W)U-\overline g(U,W)V\bigr).
\]
Thus, \(c=1,0,-1\) correspond to the de~Sitter, Minkowski and anti-de~Sitter cases, respectively.

For a constant vector \(A\in\mathbb{V}\), define
\begin{align}
    &s_A:=\langle A,\xi\rangle_{\mathbb{V}}, \quad
    m_A:=\langle A,E_c \rangle_{\mathbb{V}}, \label{eq:s_A}\\
    & \rho_A:=m_A+cs_A. \label{eq:rho_A}
\end{align}
Then \(d\rho_A=c\,ds_A\). 
Moreover, for \(X,Y\in T\mathbb Q^{n+1}_1(c)\),
\begin{align*}
\overline\nabla^2 s_A(X,Y)
&=X(Y(s_A))-(\overline\nabla_XY)(s_A)\\
&=X\langle A,Y\rangle_{\mathbb V}
  -\langle A,\overline\nabla_XY\rangle_{\mathbb V}\\
&=\langle A,D_XY-\overline\nabla_XY\rangle_{\mathbb V}\\
&=-\rho_A\,\overline g(X,Y).
\end{align*}
Thus,
\begin{equation}\label{eq:Hessians_A}
\overline\nabla^2s_A=-\rho_A\overline g.
\end{equation}
Define
\begin{equation}\label{def:general-XA}
    \mathcal{X}_A:=\overline{\nabla}s_A
    =A-\rho_A\xi-s_AE_c.
\end{equation}
\(\mathcal X_A\) is the orthogonal projection of \(A\) onto
\(T\mathbb{Q}_1^{n+1}(c)\). It is a closed conformal Killing field
satisfying
\begin{equation}\label{eq:general-connection-XA}
 \overline{\nabla}_Z\mathcal X_A=-\rho_AZ, \quad
 \text{for}\  Z\in T\mathbb{Q}_1^{n+1}(c).
\end{equation}

For another constant vector \(B\in\mathbb V\), define
\begin{equation}\label{def:general-K}
\mathcal K_{A,B}
:=
\rho_A\mathcal X_B-\rho_B\mathcal X_A.
\end{equation}
By \eqref{eq:general-connection-XA},
\begin{equation}\label{compKAB}
\overline\nabla_Z\mathcal K_{A,B}
=
c\bigl(
\langle Z,A \rangle_{\mathbb V} \mathcal X_B
-
\langle Z,B\rangle_{\mathbb V}\mathcal X_A
\bigr).
\end{equation}
The bilinear form
\((Z,W)\mapsto\overline g(\overline\nabla_Z\mathcal K_{A,B},W)\)
is therefore skew-symmetric. Hence, \(\mathcal K_{A,B}\) is a Killing
field on \(\mathbb Q^{n+1}_1(c)\).

We next describe the support hypersurfaces. Let \(A\in\mathbb V\) be a
constant timelike vector linearly independent of \(E_c\).
Set
\[
    \mathcal P_A:=
    \{\xi\in\mathbb Q^{n+1}_1(c)\mid s_A(\xi)=0\}.
\]
Along \(\mathcal{P}_A\), we have
\[
    \overline{g}(\overline{\nabla}s_A,\overline{\nabla}s_A)
    =\langle A,A \rangle_{\mathbb V}<0.
\]
Hence, \(\mathcal{P}_A\) is a spacelike hypersurface.
Choosing its unit timelike normal
\begin{equation}\label{eq:general-support-normal}
N:=-\frac{1}{\lambda_A}\mathcal{X}_A,
\quad 
\lambda_A:=\sqrt{-\langle A,A \rangle_{\mathbb V}},
\end{equation}

and applying \eqref{eq:general-connection-XA}, we obtain
\[
    \overline\nabla_ZN
    =-\frac{1}{\lambda_A}\overline\nabla_Z\mathcal X_A
    =\frac{\rho_A}{\lambda_A}Z.
\]
Since \(\rho_A=m_A+cs_A=m_A\) on \(\mathcal P_A\), the support
\(\mathcal P_A\) is totally umbilical, with constant principal curvature
\[
\kappa_A=\frac{m_A}{\lambda_A}.
\]
By the Gauss equation, its sectional curvature is
\[
    \operatorname{Sec}(\mathcal{P}_A)=c-\kappa_A^2=c-\frac{m_A^2}{\lambda_A^2}.
\]

For later use, we also introduce the auxiliary conformal Killing vector field
\[
    \mathfrak{C}_{A,B}:=s_A\mathcal{X}_B-s_B\mathcal{X}_A.
\]
Equations \eqref{eq:general-connection-XA} and \eqref{eq:rho_A} give
\[
\mathcal{L}_{\mathfrak{C}_{A,B}}\overline{g}
=2 (s_B\rho_A -s_A\rho_B )\overline{g}.
\]
Although \(\mathfrak C_{A,B}\) is not Killing in general, we define
\begin{equation}\label{def:general-T}
\mathcal T_{A,B}
:=
\langle A,B \rangle_{\mathbb V}\mathcal X_A
-\langle A,A \rangle_{\mathbb V} \mathcal X_B
+m_A\mathfrak{C}_{A,B}.
\end{equation}

On \(\mathcal P_A\), we have \(s_A=0\), and hence
\[
\begin{aligned}
\mathcal T_{A,B}
&=
\bigl(\langle A,B \rangle_{\mathbb V} -m_As_B\bigr)\mathcal X_A
-\langle A,A \rangle_{\mathbb V} \mathcal X_B \\
&=
\overline g(\mathcal X_A,\mathcal X_B)\mathcal X_A
+\lambda_A^2\mathcal X_B .
\end{aligned}
\]
Since \(N=-\lambda_A^{-1}\mathcal X_A\), this is equivalent to
\begin{equation}\label{eq:T-projection-support}
\mathcal T_{A,B}
=
\lambda_A^2(\mathcal X_B)^{\mathcal{P}_A}
\quad
\text{on }\mathcal P_{A},
\end{equation}
where \((\mathcal X_B)^{\mathcal{P}_A}\) denotes the orthogonal projection of
\(\mathcal X_B\) onto \(T\mathcal{P}_A\).  In particular,
\begin{equation}\label{orthogonal}
\overline g(\mathcal T_{A,B},N)=0
\quad
\text{on }\mathcal{P}_A.
\end{equation}
Thus, \(\mathcal{T}_{A,B}\) is tangent to \(\mathcal{P}_A\).

For \(u,v,r,s\in\mathbb V\), we use the induced bilinear form on
\(\Lambda^2\mathbb V\) defined by
\begin{equation}\label{eq:bivector-pairing}
\langle u\wedge v,r\wedge s\rangle_{\mathbb V}
:=\langle u,r\rangle_{\mathbb V}\langle v,s\rangle_{\mathbb V}
-\langle u,s\rangle_{\mathbb V}\langle v,r\rangle_{\mathbb V}.
\end{equation}

The field \(\mathcal T_{A,B}\) is a support-tangent conformal Killing
field on \(\mathbb Q_1^{n+1}(c)\) satisfying
\begin{equation}\label{eq:T-minus-tauK-conformal}
\mathcal{L}_{\mathcal{T}_{A,B}}\overline{g}=2\psi_{A,B}\overline{g},
\end{equation}
where the potential \(\psi_{A,B}\) is given by 
\begin{equation}\label{def:general-psi}
\begin{aligned}
\psi_{A,B}
& =
\langle A,A \rangle_{\mathbb V} \, \rho_B
-\langle A,B \rangle_{\mathbb V}\,\rho_A
+m_A(s_B\rho_A-s_A\rho_B) \\
& =\langle A\wedge B,(A-m_A\xi)\wedge(E_c+c\xi)\rangle_{\mathbb V}.
\end{aligned}
\end{equation}
It follows from \eqref{eq:Hessians_A} that
\begin{equation}\label{eq:Hessipsi_{AB}}
    \overline{\nabla}^2\psi_{A,B}=-c\psi_{A,B}\overline{g}.
\end{equation}

\subsection{The support two-plane}

The geometry of the support is encoded by
\[
   \Pi_A=\operatorname{span}\{A,E_c\}\subset \mathbb V.
\]
With respect to the ordered basis \((A,E_c)\), the Gram matrix of
\(\Pi_A\) is
\[
   G_A=
   \begin{pmatrix}
     -\lambda_A^2&m_A\\
     m_A&-c
   \end{pmatrix}.
\]
Hence,
\begin{equation}\label{eq:det-support-plane}
   \det G_A
   =c\lambda_A^2-m_A^2
   =\lambda_A^2\operatorname{Sec}(\mathcal P_A).
\end{equation}

\begin{proposition}\label{prop:test-signature}
When \(\Pi_A\) is nondegenerate, set
\(\mathbb E_A=\Pi_A^\perp\).
\begin{enumerate}[label=\textup{(\roman*)}]
\item If $\operatorname{Sec}(\mathcal P_A)>0$, then $\Pi_A$ is negative definite and $\mathbb E_A$ is positive definite of dimension $n+1$.
\item If $\operatorname{Sec}(\mathcal P_A)=0$, then $\Pi_A$ is degenerate.  This is the null degeneration occurring in the horospherical and flat-support cases.
\item If $\operatorname{Sec}(\mathcal P_A)<0$, then $\Pi_A$ has signature $(1,1)$ and $\mathbb E_A$ has signature $(n,1)$.
\end{enumerate}
\end{proposition}

\begin{proof}
Since \(\mathbb V\) has index two, the assertions follow from
\eqref{eq:det-support-plane} and
\(\langle A,A\rangle_{\mathbb V}<0\).
\end{proof}

\begin{remark}
The ambient construction vanishes when
\(B\in\operatorname{span}\{A,E_c\}\). Indeed,
\[
\mathcal K_{A,A}=\mathcal T_{A,A}=0,
\qquad
\psi_{A,A}=0,
\]
and the same holds for \(B=E_c\).
Thus, the natural parameter space factors through
\[
\mathbb V/\operatorname{span}\{A,E_c\},
\]
which has dimension \(n+1\).
Proposition~\ref{prop:test-signature} explains why the finite trace is positive in the round case, degenerate in the horospherical case, and indefinite in the negative-curvature cases.
\end{remark}

\subsection{The unified Minkowski-type formula}
\label{subsec:Minkowski}

Let \(\Sigma^n\subset\mathbb Q^{n+1}_1(c)\) be a compact spacelike
capillary hypersurface supported on \(\mathcal P_A\), with constant
contact angle \(\gamma\); that is,
\[
\partial\Sigma\subset\mathcal P_{A},
\qquad
\overline g(\nu,N)=-\cosh\gamma.
\]
Define 
\begin{equation}\label{def:general-test-function}
\varphi_{A,B}
:=
n\psi_{A,B}
+ H\overline g(\mathcal T_{A,B},\nu)
- n\lambda_A\cosh\gamma\,
\overline g(\mathcal K_{A,B},\nu).
\end{equation}

\begin{theorem}[Unified Minkowski-type formula]\label{thm:Minkowski-general}
Under the above assumptions, for every constant vector \(B\in\mathbb V\),
\begin{equation}\label{eq:Minkowski-general}
\int_\Sigma \varphi_{A,B}\,d\Sigma = 0.
\end{equation}
\end{theorem}

\begin{proof}

Let
\[
\mathcal T_{A,B}^{\top}
:=
\mathcal T_{A,B}
+
\overline g(\mathcal T_{A,B},\nu)\nu .
\]
Equation \eqref{eq:T-minus-tauK-conformal} gives
\[
\begin{aligned}
\operatorname{div}_{\Sigma}\mathcal T_{A,B}^{\top}
&=
\sum_{i=1}^n
\overline g(\overline\nabla_{e_i}\mathcal T_{A,B},e_i)
+
H\overline g(\mathcal T_{A,B},\nu) \\
&=
n\psi_{A,B}
+
H\overline g(\mathcal T_{A,B},\nu).
\end{aligned}
\]
On the boundary, \eqref{eq:T-projection-support} and \eqref{orthogonal} give
\[
\begin{aligned}
\overline g(\mathcal T_{A,B}^{\top},\mu)
=
\overline g(\mathcal T_{A,B},\mu) =
\cosh\gamma\,
\overline g(\mathcal T_{A,B},\overline\nu) =
\lambda_A^2\cosh\gamma\,
\overline g(\mathcal X_B,\overline\nu).
\end{aligned}
\]
Therefore, the divergence theorem gives
\begin{equation}\label{eq:Minkowski1-general}
\int_\Sigma
\left(
n\psi_{A,B}
+
H\overline g(\mathcal T_{A,B},\nu)
\right)d\Sigma
=
\lambda_A^2\cosh\gamma
\int_{\partial\Sigma}
\overline g(\mathcal X_B,\overline\nu)\,dS .
\end{equation}

Next, define
\[
\mathcal B_{A,B}
:=
\overline g(\mathcal X_A,\nu)\mathcal X_B
-
\overline g(\mathcal X_B,\nu)\mathcal X_A .
\]
Since \(\overline g(\mathcal B_{A,B},\nu)=0\), the field
\(\mathcal B_{A,B}\) is tangent to \(\Sigma\). Equation
\eqref{eq:general-connection-XA} then gives
\[
\begin{aligned}
\operatorname{div}_{\Sigma}\mathcal B_{A,B}
=
n\left(
-\rho_B\overline g(\mathcal X_A,\nu)
+
\rho_A\overline g(\mathcal X_B,\nu)
\right) =
n\overline g(\mathcal K_{A,B},\nu).
\end{aligned}
\]
For the boundary term, since
\[
\nu\wedge\mu=N\wedge\overline\nu,
\]
we have
\[
\begin{aligned}
\overline g(\mathcal B_{A,B},\mu)
&=
\overline g(\mathcal X_A\wedge\mathcal X_B,\nu\wedge\mu) \\
&=
\overline g(\mathcal X_A\wedge\mathcal X_B,N\wedge\overline\nu) \\
&=
\overline g(\mathcal X_A,N)\overline g(\mathcal X_B,\overline\nu)
-
\overline g(\mathcal X_A,\overline\nu)\overline g(\mathcal X_B,N) \\
&=
\lambda_A\overline g(\mathcal X_B,\overline\nu),
\end{aligned}
\]
where we used \(\mathcal X_A=-\lambda_AN\) and
\(\overline g(\mathcal X_A,\overline\nu)=0\) on \(\partial\Sigma\).  Applying the
divergence theorem gives
\begin{equation}\label{eq:Minkowski2-general}
\int_\Sigma
n\overline g(\mathcal K_{A,B},\nu)\,d\Sigma
=
\lambda_A
\int_{\partial\Sigma}
\overline g(\mathcal X_B,\overline\nu)\,dS .
\end{equation}

Combining \eqref{eq:Minkowski1-general} and
\eqref{eq:Minkowski2-general} proves
\eqref{eq:Minkowski-general}.
\end{proof}

\begin{lemma}\label{lem:Killing-flux-general}
For the Killing field \(\mathcal{K}_{A,B}\), one has 
\[
\int_\Sigma
H\overline g(\mathcal K_{A,B},\nu)\,d\Sigma
=
\cosh\gamma
\int_{\partial\Sigma}
\overline g(\mathcal K_{A,B},\overline\nu)\,dS
+
\frac{\sinh\gamma}{\lambda_A}
\int_{\partial\Sigma}
\psi_{A,B}\,dS.
\]
\end{lemma}

\begin{proof}
Consider the tangential projection
\[
\mathcal K_{A,B}^{\top}
=
\mathcal K_{A,B}
+
\overline g(\mathcal K_{A,B},\nu)\nu .
\]
Since \(\mathcal K_{A,B}\) is Killing, we have 
\[
\operatorname{div}_{\Sigma}\mathcal K_{A,B}^{\top}
=
H\overline g(\mathcal K_{A,B},\nu).
\]
Along \(\partial\Sigma\),  using \(\mu=\sinh\gamma\, N+\cosh\gamma\, \overline\nu\) and
\begin{equation}\label{eq:K-normal-component}
\overline g(\mathcal K_{A,B},N)
=
\frac{1}{\lambda_A}\psi_{A,B},
\end{equation}
we obtain
\[
\overline g(\mathcal K_{A,B}^{\top},\mu)
=
\frac{\sinh\gamma}{\lambda_A}\psi_{A,B}
+
\cosh\gamma\,\overline g(\mathcal K_{A,B},\overline\nu).
\]

The divergence theorem now gives the stated identity.
\end{proof}

\begin{lemma}\label{lem:test-function-boundary-value}
Under the assumptions of Theorem~\ref{thm:Minkowski-general}, the test
function \(\varphi_{A,B}\) has the following boundary value:
\begin{equation}\label{eq:general-varphi-boundary}
\varphi_{A,B}\big|_{\partial\Sigma}
= \sinh\gamma
\Bigl(
H\lambda_A^2\,\overline g(\mathcal X_B,\overline\nu)
-
n\lambda_A\cosh\gamma\,
\overline g(\mathcal K_{A,B},\overline\nu)
-
n\sinh\gamma\,\psi_{A,B}
\Bigr).
\end{equation}
\end{lemma}

\begin{proof}

Substituting the boundary identities
\[
\overline g(\mathcal T_{A,B},\nu)
= \lambda_A^2\sinh\gamma\,
\overline g(\mathcal X_B,\overline\nu),
\]
and 
\[
\overline g(\mathcal K_{A,B},\nu)
= \frac{\cosh\gamma}{\lambda_A}\psi_{A,B}
+ \sinh\gamma\,\overline g(\mathcal K_{A,B},\overline\nu)
\]
from \eqref{eq:T-projection-support} and \eqref{eq:K-normal-component}, respectively,
into \eqref{def:general-test-function} yields
the desired identity.
\end{proof}

Combining this identity with \eqref{eq:Minkowski2-general} and
Lemma~\ref{lem:Killing-flux-general}, and using that \(H\) is constant,
yields the following boundary formula.

\begin{proposition}\label{prop:Minkowski-boundary-general}
Under the same assumptions as in Theorem~\ref{thm:Minkowski-general},
\begin{equation}\label{eq:general-test-zero-boundary}
\int_{\partial\Sigma}\varphi_{A,B}\,dS=0.
\end{equation}
\end{proposition}

The conformal interpretation of the construction and its intrinsic
extension are discussed in Subsection~\ref{subsec:conformal-interpretation}.

\subsection{Model supports in de~Sitter and Minkowski spaces}

We begin with a model-reduction lemma showing that every spacelike
totally umbilical hypersurface in \(\mathbb Q^{n+1}_1(c)\) is contained
in a hypersurface \(\mathcal P_A\) of the form defined above.

\begin{lemma}\label{lem:classification_umbilical}
Let \(\mathcal P^n\subset \mathbb Q^{n+1}_1(c)\), \(n\geq 2\), be a connected spacelike totally umbilical hypersurface, with unit timelike normal \(N\) and 
\begin{equation}\label{eq:Weingarten}
\overline{\nabla}_Z N=\kappa Z, \qquad Z\in T\mathcal P.
\end{equation}
Then \(\kappa\) is constant and there exists a constant timelike vector \(A\in \mathbb V\) such that
\[\mathcal P\subset \mathcal P_A, \qquad N=-\lambda_A^{-1}\mathcal X_A, \qquad \kappa=\frac{m_A}{\lambda_A}.\]
\end{lemma}

\begin{proof}
Using \eqref{eq:Weingarten} and the Codazzi equation, we obtain
\[
\overline{R}(Z,W)N=Z(\kappa)W-W(\kappa)Z,
\]
for \(Z,W\in T\mathcal P\). Since \(\mathbb Q^{n+1}_1(c)\) has
constant sectional curvature \(c\), the left-hand side vanishes; hence,
\[
Z(\kappa)W-W(\kappa)Z=0.
\]
Since \(n\geq2\), for every nonzero \(Z\in T\mathcal P\) we may choose
\(W\in T\mathcal P\) linearly independent of \(Z\). It follows that
\(Z(\kappa)=0\). Thus, \(d\kappa=0\), and the connectedness of
\(\mathcal P\) implies that \(\kappa\) is constant.

Along \(\mathcal P\) we define
\[
A:=-N+\kappa\xi.
\]
For \(Z\in T\mathcal P\),
\[
D_Z A=-D_ZN+\kappa D_Z\xi=-\kappa Z+\kappa Z=0.
\]
Thus, \(A\) is constant. Moreover,
\[
\langle A, A\rangle_{\mathbb V}=-1, \qquad \langle A,\xi\rangle_{\mathbb V}=0,\qquad m_A=\langle A,E_c\rangle_{\mathbb V}=\kappa,
\] and therefore,
\[
\mathcal X_A=A-m_A\xi=-N.
\]
Using the convention in the light-cone model,
\[
\lambda_A:=\sqrt{-\langle A,A\rangle_{\mathbb V}}=1,
\qquad
s_A:=\langle A,\xi\rangle_{\mathbb V}=0.
\]
Consequently,
\[
N=-\mathcal X_A,
\qquad
\kappa=m_A,
\qquad
\mathcal P\subset\mathcal P_A.
\]
This completes the proof.
\end{proof}

We now specialize the light-cone construction to the de~Sitter and
Minkowski models appearing in Theorem~\ref{thm:intro-criterion}.
In the de~Sitter and Minkowski models,
\(\langle\cdot,\cdot\rangle\) denotes the corresponding ambient
pseudo-Euclidean inner product; the subscript \(\mathbb V\) is retained for
the light-cone ambient space.

\subsubsection{De~Sitter sections}
\label{sec:de-sitter}
Choose \(E=E_1\in\mathbb V\) with
\[
\langle E,E \rangle_{\mathbb V}=-1.
\]
Then \(E^\perp\) is naturally identified with
\((\mathbb R^{n+1,1}, \langle \cdot , \cdot \rangle)\).
The map 
\[x\mapsto \xi=x-E\]
identifies
\[
\mathbb S^{n+1}_1
=
\{x\in E^\perp: \langle x,x\rangle=1\}
\]
with \(\mathbb Q^{n+1}_1(1)\).

Let \(a,b\in E^\perp\) and \(\tau\in\mathbb R\), and set
\[
A=a-\tau E,
\qquad
B=b.
\]
Assume that
\[
\tau^2>\langle a,a\rangle,
\]
so that \(A\) is timelike. 
Then
\[s_A=\langle x,a\rangle -\tau, \qquad m_A=\tau, \qquad \rho_A=\langle x,a\rangle,\]
and
\[
\mathcal{P}_{a,\tau}=\{x\in \mathbb S^{n+1}_1| \langle x,a\rangle =\tau\}.
\]
Moreover,
\[
\lambda_{a,\tau}=\sqrt{\tau^2-\langle a,a\rangle},
\qquad
\operatorname{Sec}(\mathcal{P}_{a,\tau}) = -\frac{\langle a,a\rangle}{\lambda_{a,\tau}^2}.
\]
For later use, we record the standard ambient fields
\[
X_v:=v-\langle x,v\rangle x, \qquad v\in E^\perp.
\]
Under the above identification, the fields constructed in Subsection \ref{subsec:light-cone} become
\[
K_{a,b}=\langle x,a\rangle b-\langle x,b\rangle a,
\]
\[
    T_{a,b,\tau}
    =\langle a,b\rangle X_a-\langle a,a\rangle X_b 
    +\tau K_{a,b},
\]
and
\[
\psi_{A,B}
=
\langle a\wedge x,a\wedge b\rangle
=\inn{a}{a} \inn{x}{b}-\inn{a}{b}\inn{x}{a}.
\]

The causal type of \(a\) determines the corresponding geometry of
\(\mathcal P_{a,\tau}\).

\smallskip
\paragraph{\textbf{Round supports.}}
    If \(a\) is timelike, 
    we normalize \(\langle a,a\rangle=-1\).
    Then
    \[
        \lambda_{a,\tau}=\sqrt{1+\tau^2}, \quad
        \operatorname{Sec}(\mathcal P_{a,\tau})=\lambda_{a,\tau}^{-2}>0.
    \]
Furthermore, \(\Pi_A=\operatorname{span}\{a,E\}\) is negative definite and
\(\mathbb E_A:=\Pi_A^\perp=a^\perp\cap E^\perp\) is a canonically defined
positive-definite space of dimension \(n+1\).

\smallskip
\paragraph{\textbf{Horospherical supports.}}   
    If \(a\) is lightlike, we assume \(\tau>0\). Then
    \[
        \lambda_{a,\tau}=\tau, \quad
        \operatorname{Sec}(\mathcal P_{a,\tau})=0.
    \]
    In this case, \( \Pi_A
    =
    \operatorname{span}\{a,E\}\)
is degenerate, with \(\operatorname{Rad}(\Pi_A)
    =
    \mathbb Ra\).
Thus, the parameter module has a null degeneration.  The corresponding
effective one-dimensional test direction will be identified in
Section~\ref{subsec:finite-family}.    
    
\noindent

If \(a\) is spacelike, we normalize
\(\langle a,a\rangle=1\). The condition that \(A\) be timelike is then
\(\tau^2>1\), and
\[
\operatorname{Sec}(\mathcal P_{a,\tau})
=
-\frac{1}{\tau^2-1}<0.
\]
The parameter module is indefinite and contains one timelike
direction.  Since these supports do not enter the main rigidity
theorem, their detailed formulas are deferred to
Section~\ref{sec:negative-curvature}.

\subsubsection{Spacelike hyperplanes in Minkowski space}
\label{sec:hyperplane}

Choose null vectors \(E_0,E_0^*\in\mathbb V\) such that
\[
\langle E_0,E_0\rangle_{\mathbb V}
=
\langle E_0^*,E_0^*\rangle_{\mathbb V}
=0,
\qquad
\langle E_0,E_0^*\rangle_{\mathbb V}=1.
\]
We identify
\((\mathbb R^{n,1},\langle\cdot,\cdot\rangle)\)
with \(\{E_0,E_0^*\}^\perp\) and use the standard embedding
\[
x\longmapsto
\xi=x+E_0^*-\frac12\langle x,x\rangle E_0.
\]

Since translations of Minkowski space are ambient isometries, after a
translation we may assume that the spacelike totally geodesic support is
\[
\mathcal P
=
\{x\in\mathbb R^{n,1}\mid \langle x,N\rangle=0\},
\qquad
\langle N,N\rangle=-1.
\]
It is represented by
\[
A=-N.
\]
Indeed, 
\[
s_A=-\langle x,N\rangle,
    \qquad
    m_A=0,
    \qquad
    \lambda_A=1.
\]
Consequently,
\(
  \operatorname{Sec}(\mathcal P_A)=0,
\)
and
\(
    \Pi_A
    =
    \operatorname{span}\{N,E_0\}
\)
is degenerate, with
\[
    \operatorname{Rad}(\Pi_A)
    =
    \mathbb RE_0.
\]

The three support geometries relevant to the main theorem may therefore
be summarized as follows:
\[
\begin{array}{c|c|c|c}
\text{support}
&
\Pi_A
&
\operatorname{Sec}(\mathcal P_A)
&
\text{parameter type}
\\ \hline
\text{round de~Sitter}
&
\operatorname{span}\{a,E\}
&
>0
&
\text{positive definite}
\\
\text{de~Sitter horosphere}
&
\operatorname{span}\{a,E\}
&
0
&
\operatorname{Rad}(\Pi_A)=\mathbb Ra
\\
\text{Minkowski hyperplane}
&
\operatorname{span}\{N,E_0\}
&
0
&
\operatorname{Rad}(\Pi_A)=\mathbb RE_0.
\end{array}
\]

In particular, the round case admits a canonical positive parameter
space, whereas the two zero-curvature cases possess only a canonical
null flag.  

\section{The Jacobi--Robin module}
\label{sec:Jacobi-Robin}

\subsection{The support-function identity}

We first establish the support-function identity for a conformal Killing
field and then derive the Jacobi action and Robin boundary condition for
the unified test function \(\varphi_{A,B}\) defined in
\eqref{def:general-test-function}.

\begin{lemma}
\label{Lemma4.1}
Let \((\overline M^{n+1},\overline g)\) be a Lorentzian manifold, and let
\(x\colon(\Sigma^n,g)\to(\overline M^{n+1},\overline g)\) be a spacelike
immersion with unit timelike normal \(\nu\). Suppose that \(X\) is a
conformal Killing field on \(\overline M\) satisfying
\(\mathcal L_X\overline g=2\psi\,\overline g\). Then
        \begin{equation} \label{eq:Jacobi_CK}
            J\big(\overline g(X,\nu)\big)=\overline g(X^\top,\nabla^\Sigma H)+H\psi-n\nu(\psi),
        \end{equation}
        where \(J=\Delta-|h|^2-\overline{\mathrm{Ric}}(\nu,\nu)\).
        In particular, if \(H\) is constant, then 
        \[J\big(\overline g(X,\nu)\big)=H\psi-n\nu(\psi).\]      
    \end{lemma}

\begin{proof}
Let \(\Phi_t\) be the local flow generated by \(X\), and consider the
variation \(x_t=\Phi_t\circ x\).
The standard first variation formula for mean curvature gives
\begin{equation} \label{eq:HtCK}
        \frac{d}{dt}\Big|_{t=0}H_t=-J\overline g(X,\nu)+\overline g(X^\top,\nabla^\Sigma H).
\end{equation}        
Since \(X\) is conformal Killing, 
\[\frac{d}{dt}\Phi_t^*\overline g
=\Phi_t^*(\mathcal L_X\overline g)
=2(\psi\circ\Phi_t)\Phi_t^*\overline g.\]
Consequently,
\[
    \Phi_t^* \overline g=e^{2u_t} \overline g, \qquad u_0=0, \qquad \frac{\partial u_t}{\partial t}\Big|_{t=0}=\psi.
\]
Under the conformal change \(\overline g_t=e^{2u_t}\overline g\), the
mean curvature transforms as
\[
 H_t=e^{-u_t}(H+n\nu(u_t)).
\]
Differentiating at \(t=0\) gives
\begin{equation}\label{eq:dHdt}
    \frac{d}{dt}\Big|_{t=0}H_t=-H\psi+n\nu(\psi).
\end{equation}
Comparing \eqref{eq:HtCK} and \eqref{eq:dHdt} proves \eqref{eq:Jacobi_CK}.
\end{proof}

From now on, we return to the setting of
Subsection~\ref{subsec:Minkowski}. Namely, let
\[
\Sigma^n\subset\mathbb Q_1^{n+1}(c)
\]
be a compact spacelike capillary hypersurface supported on
\(\mathcal P_A\).
Since
\[
\overline{\operatorname{Ric}}=cn\overline g,
\qquad
\overline g(\nu,\nu)=-1,
\]
the Jacobi operator reduces to
\[
J=\Delta-|h|^2+cn.
\]

\begin{proposition}
\label{prop:jacobi-action}
For every \(B\in \mathbb V\),
\begin{equation}\label{eq:J-varphi-AB}
J\varphi_{A,B}
=
(H^2-n|h|^2)\psi_{A,B}.
\end{equation}
\end{proposition}

\begin{proof}
We first compute \(J\psi_{A,B}\). For \(Y,Z\in T\Sigma\), the Gauss
formula gives
\begin{align*}
\nabla^2 \psi_{A,B}(Y,Z)&=\overline{\nabla}^2\psi_{A,B}(Y,Z)+h(Y,Z)\nu(\psi_{A,B})\\
&=-c\psi_{A,B} g(Y,Z) + h(Y,Z)\nu(\psi_{A,B}),
\end{align*}
where we used \eqref{eq:Hessipsi_{AB}}.
Taking the trace with respect to \(g\), we obtain
\[
\Delta \psi_{A,B}=-cn\psi_{A,B}+H\nu(\psi_{A,B}).
\]
Consequently,
\begin{equation}\label{eq:Jpsi_{AB}}
    J\psi_{A,B}=-|h|^2\psi_{A,B}+H\nu(\psi_{A,B}).
\end{equation}

Next, \(\mathcal T_{A,B}\) is conformal Killing with conformal potential
\(\psi_{A,B}\). Since \(H\) is constant,
Lemma~\ref{Lemma4.1} gives
\begin{equation}\label{eq:JT}
J\overline{g}(\mathcal{T}_{A,B}, \nu)=H\psi_{A,B}- n \nu(\psi_{A,B}).
\end{equation}

Since \(\mathcal{K}_{A,B}\) is Killing,
\begin{equation}\label{eq:Jphi_{AB}}
J\overline{g}(\mathcal{K}_{A,B}, \nu)=0.
\end{equation}
Combining \eqref{eq:Jpsi_{AB}}, \eqref{eq:JT}, and
\eqref{eq:Jphi_{AB}} with the definition of \(\varphi_{A,B}\), we obtain
\begin{align*}
  J\varphi_{A,B}&=nJ\psi_{A,B}+HJ\overline{g}(\mathcal{T}_{A,B}, \nu)-n\lambda_A\cosh\gamma J\overline{g}(\mathcal{K}_{A,B},\nu)\\
  &= n\bigl( -|h|^2\psi_{A,B} +H\nu(\psi_{A,B})\bigr) 
+H\bigl( H\psi_{A,B}-n\nu(\psi_{A,B})\bigr) \\
  &=(H^2-n|h|^2)\psi_{A,B}.  \qedhere
\end{align*}
\end{proof}

\subsection{The Robin boundary condition}
\label{subsec:robin-boundary-condition}
Since
\(h^{\mathcal P_A}
=
\kappa_A\,\overline g|_{T\mathcal P_A}\)
and \(\overline g(\overline\nu,\overline\nu)=1\), the Robin coefficient
in Lemma~\ref{lem:second-variation} becomes
\begin{equation}\label{eq:qA}
q=q_A
:=
\coth\gamma\,h(\mu,\mu)
-
\frac{\kappa_A}{\sinh\gamma}.
\end{equation}

\begin{lemma}
Let \(Y\) be a conformal Killing field on \((\overline M,\overline g)\), and set
\[
f_Y=\overline g(Y,\nu),
\qquad
b_Y=\overline g(Y,N).
\]
Then, along \(\partial\Sigma\),
    \begin{equation}\label{lem:Robin-boundary-identity}
    (\nabla_{\mu}-q_A)f_Y=\overline{\nu}(b_Y)+\frac{\kappa_A\cosh \gamma-h(\mu,\mu)}{\sinh\gamma}\, b_Y.
    \end{equation}
In particular, if \(Y\) is tangent to the totally umbilical support
\(\mathcal P_A\), then \((\nabla_{\mu}-q_A)f_Y=0\).
\end{lemma}

\begin{proof}
    From \eqref{eq:boundary-frame-inverse},  
    \begin{equation}\label{eq:two-supports}
    \overline{g}(Y,\overline\nu)=\frac{f_Y-\cosh\gamma\, b_Y}{\sinh\gamma},\qquad
    \overline{g}(Y,\mu)=\frac{\cosh\gamma \,f_Y- b_Y}{\sinh\gamma}.
    \end{equation}
    Since \(Y\) is conformal Killing, the pure-trace terms in the
    \((N,\overline\nu)\)-component cancel; hence,
    \begin{equation}\label{eq:D-conformal-Killing}
    \overline{g}(\overline\nabla_{\mu}Y, \nu)=\overline{g}(\overline\nabla_{\overline \nu}Y, N).
    \end{equation}
    Since \(\mathcal{P}_A\) is totally umbilical, 
    \[\overline{\nabla}_{\overline \nu}N=\kappa_A\overline{\nu},\]
and thus
\begin{equation}\label{eq:D-YN}
\overline{g}(\overline\nabla_{\overline\nu}Y, N)=\overline{\nu}(b_Y)-\kappa_A \overline{g}(Y,\overline\nu).
\end{equation}
Combining \eqref{eq:D-conformal-Killing}, \eqref{eq:D-YN}, and the
principal-direction identity \eqref{eq:Dnu}, we obtain
\begin{align*}
\nabla_{\mu}f_Y&=\overline{g}(\overline\nabla_{\mu}Y, \nu)+\overline{g}(Y, \overline\nabla_{\mu}\nu)\\
&=\overline{\nu}(b_Y)-\kappa_A\overline{g}(Y,\overline\nu)+h(\mu,\mu)\overline{g}(Y,\mu).
\end{align*}
Substituting \eqref{eq:two-supports} proves
\eqref{lem:Robin-boundary-identity}.
\end{proof}

\begin{lemma}\label{lem:Npsi}
On $\mathcal P_A$, one has
\begin{equation}\label{eq:Npsi}
N(\psi_{A,B})=\kappa_A\psi_{A,B}.
\end{equation}
\end{lemma}

\begin{proof}
Differentiating the expression
\[
\psi_{A,B}
=\langle A,A\rangle_{\mathbb V}\rho_B
-\langle A,B\rangle_{\mathbb V}\rho_A
+m_A(s_B\rho_A-s_A\rho_B)
\]
gives
\[
Y(\psi_{A,B})=(c\langle A,A\rangle_{\mathbb V}+m_A^2)\overline{g}(\mathcal{X}_B, Y)
-(c \langle A,B\rangle_{\mathbb V}+m_Am_B) \overline{g}(\mathcal{X}_A, Y).
\]
Taking \(Y=N\) and using
\(N=-\lambda_A^{-1}\mathcal X_A\) on \(\mathcal P_A\), together with
\begin{align*}
&\overline{g}(\mathcal{X}_A, N)=\lambda_A, \qquad 
\overline{g}(\mathcal{X}_B,N)=-\lambda_A^{-1}(\langle A,B\rangle_{\mathbb V}-m_A s_B), \\
& s_A=0,\qquad \rho_A=m_A, \qquad \rho_B=m_B+c s_B,
\end{align*}
we obtain
\[N(\psi_{A,B})=\frac{m_A}{\lambda_A}\psi_{A,B}=\kappa_A\psi_{A,B} \qquad \text{ on } \mathcal{P}_A.\]
\end{proof}

We next verify the boundary condition.
    
\begin{proposition}\label{prop:Robin-unified-stable}
For every \(B\in\mathbb V\), 
\[
\nabla_\mu\varphi_{A,B}
=
q_{A}\varphi_{A,B} \qquad \text{on } \, \partial \Sigma.
\]
\end{proposition}

\begin{proof}
We first apply \eqref{lem:Robin-boundary-identity} to the conformal Killing field \(Y=\mathcal{T}_{A,B}\). 
Since \(\mathcal{T}_{A,B}\) is tangent to \(\mathcal{P}_A\), \(b_Y=0\) on \(\mathcal P_A\), and hence
\[b_{\mathcal{T}_{A,B}}=0,\qquad \overline{\nu}(b_{\mathcal{T}_{A,B}})=0,\qquad \text{on } \partial\Sigma.\]
Thus,
\begin{equation}\label{eq:Robin_T_AB}
(\nabla_{\mu}-q_A)\overline{g}(\mathcal{T}_{A,B},\nu)=0.
\end{equation}

Next, apply \eqref{lem:Robin-boundary-identity} to the Killing field
\(Y=\mathcal K_{A,B}\). The normal-component formula
\eqref{eq:K-normal-component} gives
\begin{equation}\label{eq:boundary-K-unified}
(\nabla_{\mu}-q_A) \overline{g}(\mathcal{K}_{A,B},\nu)=\frac{1}{\lambda_A}\parens*{\overline\nu(\psi_{A,B})
+\frac{\kappa_A\cosh\gamma -h(\mu,\mu)}{\sinh\gamma}\psi_{A,B}}.
\end{equation}

Using \(\mu=\sinh\gamma\,N+\cosh\gamma\,\overline\nu\) and
\eqref{eq:Npsi}, we obtain
\begin{align*}
\nabla_{\mu}\psi_{A,B}
&=\sinh\gamma\,N(\psi_{A,B})
  +\cosh\gamma\,\overline\nu(\psi_{A,B})\\
&=\sinh\gamma\,\kappa_A\psi_{A,B}
  +\cosh\gamma\,\overline\nu(\psi_{A,B}).
\end{align*}
Thus,
\begin{align}
(\nabla_{\mu}-q_A)\psi_{A,B}&=\cosh\gamma \, \overline{\nu}(\psi_{A,B})
+\parens*{\sinh\gamma \kappa_A -\coth\gamma h(\mu,\mu)+\frac{\kappa_A}{\sinh\gamma}}\psi_{A,B}\nonumber\\
&=\cosh\gamma\parens*{\overline{\nu}(\psi_{A,B})+\frac{\kappa_A\cosh\gamma-h(\mu,\mu)}{\sinh\gamma}\psi_{A,B}}\label{eq:Robin_psi}
\end{align}
Comparing \eqref{eq:boundary-K-unified} and \eqref{eq:Robin_psi}, we obtain
\begin{equation}\label{eq:Robin_psi-TAB}
(\nabla_{\mu}-q_A)(\psi_{A,B}-\lambda_A\cosh\gamma \overline{g}(\mathcal{K}_{A,B}, \nu))=0.
\end{equation}
Finally, the definition \eqref{def:general-test-function}, together with 
\eqref{eq:Robin_T_AB} and \eqref{eq:Robin_psi-TAB}, gives
\[(\nabla_{\mu}-q_A)\varphi_{A,B}=0.\]
This completes the proof.
\end{proof}

Theorem~\ref{thm:Minkowski-general} and
Propositions~\ref{prop:jacobi-action} and
\ref{prop:Robin-unified-stable} yield the following Jacobi--Robin
package.

\begin{theorem}\label{thm:test-package}
For every \(B\in\mathbb V\), the function \(\varphi_{A,B}\) satisfies
\[
   \int_\Sigma\varphi_{A,B}\,d\Sigma=0,
\]
\[
   J\varphi_{A,B}=-(n|h|^2-H^2)\psi_{A,B}
   \qquad\text{in }\Sigma,
\]
and
\[
   (\nabla_\mu-q_A)\varphi_{A,B}=0
   \qquad\text{on }\partial\Sigma.
\]
Consequently,
\begin{equation}\label{eq:Qphi}
   Q(\varphi_{A,B},\varphi_{A,B})
   =-\int_\Sigma
       (n|h|^2-H^2)\varphi_{A,B}\psi_{A,B}\,d\Sigma.
\end{equation}
\end{theorem}

\begin{remark}
The Robin boundary condition is not an incidental cancellation. Indeed,
the second-variation formula in
Theorem~\ref{thm:second-var-unified} identifies
\(\nabla_\mu-q_A\) as the linearization of the capillary contact-angle
condition. Moreover, the two basic Robin terms
\[
\overline g(\mathcal T_{A,B},\nu)
\quad\text{and}\quad
\psi_{A,B}-\lambda_A\cosh\gamma\,\overline g(\mathcal K_{A,B},\nu)
\]
lie in the kernel of \(\nabla_\mu-q_A\) separately.
\end{remark}

\section{Finite-trace identities for nonnegative-curvature supports}
\label{sec:finite-trace}

\subsection{Finite test families}
\label{subsec:finite-family}

We now specialize the parameter $B$ in the three nonnegative-curvature support geometries.

\subsubsection{Round supports in de~Sitter space}
Write
\[
\mathbb S^{n+1}_1
=\{x\in\mathbb R^{n+1,1}\mid\langle x,x\rangle=1\}.
\]
Let
\[
\mathcal P_{a,\tau}
=\{x\in\mathbb S^{n+1}_1\mid\langle x,a\rangle=\tau\},
\]
where \(\langle a,a\rangle=-1\). Set
\(\lambda=\sqrt{1+\tau^2}\).
Choose an orthonormal spacelike basis
\(\{b_1,\ldots,b_{n+1}\}\) of \(a^\perp\).  
For each \(\alpha\),
\[
\psi_{\alpha}
= -\langle x,b_\alpha\rangle,
\]
and
\[
\begin{aligned}
\varphi_{\alpha}
&=
-n\langle x,b_\alpha\rangle
+
H\langle X_{b_\alpha}+\tau K_{a,b_\alpha},\nu\rangle 
-
n\lambda \cosh\gamma\,
\langle K_{a,b_\alpha},\nu\rangle
\end{aligned}
\]
where
\[X_{b_\alpha}=b_\alpha-\inn{x}{b_\alpha}x, \qquad K_{a,b_\alpha}=\inn{x}{a}b_\alpha-\inn{x}{b_\alpha}a.\]
The finite trace in the round case is taken over the \(n+1\) functions
\[
\{\varphi_1,\ldots,\varphi_{n+1}\}.
\]

\subsubsection{Horospherical supports in de~Sitter space}

Let 
\[
\mathcal P_{a,\tau}=\{x\in \mathbb{S}^{n+1}_1 | \inn{x}{a}=\tau\}, \qquad 
\inn{a}{a}=0,\qquad
\tau>0.
\]
Choose \(b\in\mathbb{R}^{n+1,1}\) such that
\[
\langle b,b \rangle=-1,
\qquad
\langle a,b \rangle=1.
\]
Then
\[
\lambda=\tau,
\qquad
\psi
=
-\langle x,a \rangle,
\]
and the corresponding test function is
\[
\begin{aligned}
\varphi
&=
-n \langle  x,a\rangle 
+
H \langle  X_a+\tau K_{a,b},\nu \rangle 
-
n\tau\cosh\gamma\,
\langle K_{a,b},\nu \rangle .
\end{aligned}
\]
In the horospherical case, the finite-trace contribution
reduces to \(Q(\varphi,\varphi)\). It will be supplemented by the
Dirichlet Green correction introduced in
Subsection~\ref{subsec:finite-trace}.

\subsubsection{Flat supports in Minkowski space}

Let
\[
\mathcal P
=
\{x\in\mathbb R^{n,1}| \langle x,N \rangle =0\},
\qquad
\langle N,N \rangle =-1,
\]
where \(N\) is future-directed.  
In the flat light-cone model, this choice corresponds to
\[A=-N, \qquad B=-E_0^*.\]
For this choice,
\(
\psi_0=1,
\)
and the canonical test function is
\[
\varphi_0
=
n
+
H\inn{x}{\nu}
+
n\cosh\gamma\,\inn{N}{\nu}.
\]

In the flat case, the finite trace reduces to
\(Q(\varphi_0,\varphi_0)\).

\subsection{The finite-trace identity}
\label{subsec:finite-trace}

We now derive the finite-trace identity used in the rigidity argument.
It expresses the Jacobi--Robin test-family contribution as an integral
of the trace-free second fundamental form against a nonnegative
geometric weight. The umbilicity defect satisfies
\[
n|h|^2-H^2\ge 0.
\]
It vanishes identically if and only if \(\Sigma\) is totally umbilical.

\subsubsection{The de~Sitter Green potential}

Let
\[
\mathcal P_{a,\tau}=\{x\in \mathbb S^{n+1}_1| \langle x,a\rangle=\tau\}
\]
be either a round or a horospherical support in de~Sitter space. Thus,
\[
\langle a,a\rangle=-1
\]
in the round case, whereas
\[
\langle a,a\rangle=0,
\qquad
\tau>0,
\]
in the horospherical case.
Put
\[
\lambda=\lambda_{a,\tau}:=\sqrt{\tau^2-\langle a,a\rangle},
\qquad
s:=\langle x,a\rangle,
\qquad
u:=\langle X_a,\nu\rangle,
\]
where
\[
X_a=a-\langle x,a\rangle x.
\]
We use the Dirichlet-normalized Green potential
\begin{equation}\label{2auxdesitter}
\Phi_a
:=
n\bigl(\langle X_a,\nu\rangle-\lambda\cosh\gamma\bigr)
-
H\bigl(\langle x,a\rangle-\tau\bigr).
\end{equation}

Equivalently,
\[
\Phi_a=n(u-\lambda \cosh \gamma)-H(s-\tau).
\]

\begin{lemma}
\label{lem:green-potential}
The function \(\Phi_a\) satisfies
\[
\Phi_a=0\quad\text{on }\partial\Sigma,
\]
and
\[
\Delta\Phi_a=(n|h|^2-H^2)\,\langle X_a,\nu\rangle
=(n|h|^2-H^2)\,u
\quad\text{in }\Sigma.
\]
Consequently,
\[
\mathcal G_a
:=
\int_\Sigma \Phi_a\Delta\Phi_a\,d\Sigma
=
-\int_\Sigma |\nabla\Phi_a|^2\,d\Sigma
\le 0.
\]
\end{lemma}

\begin{proof}
On \(\partial\Sigma\), we have \(s=\tau\). Moreover, along
\(\mathcal P_{a,\tau}\),
\[
X_a=-\lambda N.
\]
Since \(\inn{\nu}{N}=-\cosh\gamma\), it follows that
\[
u=\langle X_a,\nu\rangle=\lambda\cosh\gamma
\quad\text{on }\partial\Sigma.
\]
Hence, \(\Phi_a=0\) on \(\partial\Sigma\).

It remains to compute the Laplacian. Since \(X_a=\overline\nabla s\) and
\[
\overline\nabla_ZX_a=-sZ,
\]
we have, on \(\Sigma\),
\[
\Delta s=-ns+Hu.
\]
Also, by the conformal support-function identity applied to \(X_a\),
\[
\Delta u=|h|^2u-Hs.
\]
Therefore,
\[
\Delta(nu-Hs)
=
n(|h|^2u-Hs)-H(-ns+Hu)
=
(n|h|^2-H^2)u.
\]
Since \(\Phi_a=nu-Hs+\text{constant}\), this proves the asserted
Laplacian identity. The boundary condition
\(\Phi_a|_{\partial\Sigma}=0\) then gives the Green identity.
\end{proof}

\subsubsection{Statement of the finite-trace identity}

Define the weight \(F\) by
\begin{equation} \label{eq:F}
F
=
\begin{cases}
n\inn{X_a^\top}{X_a^\top},
&
\text{in the round and horospherical de~Sitter cases},\\[4pt]
-n\bigl(1+\cosh\gamma\,\inn{N}{\nu}\bigr),
&
\text{in the flat Minkowski hyperplane case}.
\end{cases}
\end{equation}

\begin{theorem}[Finite-trace identity]
\label{thm:finite-trace}
Let \(\Sigma^n\) be a compact spacelike capillary hypersurface with
constant contact angle \(\gamma>0\).

\begin{enumerate}
\item In the round de~Sitter case,
\[
\sum_{\alpha=1}^{n+1}Q(\varphi_\alpha,\varphi_\alpha)+\mathcal G_a
=
\int_\Sigma (n|h|^2-H^2)F\,d\Sigma .
\]

\item In the horospherical de~Sitter case,
\[
Q(\varphi,\varphi)+\mathcal G_a
=
\int_\Sigma (n|h|^2-H^2)F\,d\Sigma .
\]

\item In the flat Minkowski hyperplane case,
\[
Q(\varphi_0,\varphi_0)
=
\int_\Sigma (n|h|^2-H^2)F\,d\Sigma .
\]
\end{enumerate}

\end{theorem}

\begin{proof}

We prove the three cases separately.

\paragraph{\textbf{Round de~Sitter support.}}

Assume
\[
\langle a,a\rangle=-1,
\qquad
\lambda=\sqrt{1+\tau^2}.
\]
Let \(\{b_1,\ldots,b_{n+1}\}\) be an orthonormal spacelike basis of
\(a^\perp\).  For each \(\alpha\),
\[
\psi_\alpha=-\langle x,b_\alpha\rangle,
\]
and
\[
\varphi_\alpha
=
-n\langle x,b_\alpha\rangle
+
H\langle X_{b_\alpha}+\tau K_{a,b_\alpha},\nu\rangle
-
n\lambda\cosh\gamma\,\langle K_{a,b_\alpha},\nu\rangle .
\]
By the Jacobi--Robin test package,
\[
Q(\varphi_\alpha,\varphi_\alpha)
=
-\int_\Sigma (n|h|^2-H^2)\,\varphi_\alpha\psi_\alpha\,d\Sigma .
\]
The elementary contractions
\[
\sum_{\alpha=1}^{n+1}\inn{x}{b_\alpha}^2
=
1+\inn{x}{a}^2,
\qquad
\sum_{\alpha=1}^{n+1}\inn{x}{b_\alpha}X_{b_\alpha}
=
\inn{x}{a}X_a,
\]
and
\[
\sum_{\alpha=1}^{n+1}\inn{x}{b_\alpha}K_{a,b_\alpha}
=
-X_a,
\]
give
\[
\begin{aligned}
\sum_{\alpha=1}^{n+1}
\varphi_{\alpha}\psi_{\alpha}
&=
n(1+s^2)
+
H(\tau-s)u 
-
n\lambda\cosh\gamma\,u.
\end{aligned}
\]
On the other hand, by Lemma~\ref{lem:green-potential},
\[
\mathcal G_a
=
\int_\Sigma (n|h|^2-H^2)\,u
\bigl(n(u-\lambda \cosh\gamma)-H(s-\tau)\bigr)\,d\Sigma .
\]
Therefore,
\[
\begin{aligned}
\sum_{\alpha=1}^{n+1}Q(\varphi_\alpha,\varphi_\alpha)+\mathcal G_a
&=
n\int_\Sigma
 (n|h|^2-H^2)\,
(u^2-1-s^2)\,d\Sigma .
\end{aligned}
\]
Since
\(\inn{X_a}{X_a}=-1-\inn{x}{a}^2,\)
we have
\[
-1-\inn{x}{a}^2+\inn{X_a}{\nu}^2
=
\inn{X_a^\top}{X_a^\top}.
\]
Hence,
\[
\sum_{\alpha=1}^{n+1} Q(\varphi_\alpha, \varphi_\alpha) +\mathcal G_a
=
n\int_\Sigma
(n|h|^2-H^2)\,
\inn{X_a^\top}{X_a^\top}\,d\Sigma.
\]

\paragraph{\textbf{Horospherical de~Sitter support.}}
Assume
\[
\langle a,a\rangle=0,\qquad \tau>0,
\qquad \lambda=\tau.
\]
Choose \(b\in\mathbb R^{n+1,1}\) such that
\[
\langle b,b\rangle=-1,
\qquad
\langle a,b\rangle=1.
\]
Set
\[
Y:=K_{a,b},
\qquad
f_Y:=\langle Y,\nu\rangle .
\]
The canonical horospherical test function is
\[
\varphi
=
-n\inn{x}{a}+H\langle X_a+\tau Y,\nu\rangle
-
n\tau \cosh\gamma \langle Y,\nu\rangle .
\]
Thus,
\[
\varphi=-ns+Hu+\tau(H-n\cosh \gamma)\inn{Y}{\nu},
\qquad
\psi=-s.
\]
The Jacobi--Robin test package gives
\[
Q(\varphi,\varphi)
=
-\int_\Sigma  (n|h|^2-H^2) \,\varphi\psi\,d\Sigma
=
\int_\Sigma  (n|h|^2-H^2)\,s\varphi\,d\Sigma .
\]
Using Lemma~\ref{lem:green-potential}, we obtain
\[
\begin{aligned}
Q(\varphi,\varphi)+\mathcal G_a
&=
\int_\Sigma  (n|h|^2-H^2)
\bigl[-ns^2+Hsu+\tau(H-n\cosh \gamma)s f_Y \bigr]\,d\Sigma  \\
&\quad+
\int_\Sigma  (n|h|^2-H^2)
\bigl[nu^2-Hsu+\tau(H-n\cosh \gamma)u\bigr]\,d\Sigma  \\
&=
n\int_\Sigma (n|h|^2-H^2)(u^2-s^2)\,d\Sigma \\
&\qquad +
\tau(H-n\cosh\gamma)
\int_\Sigma (n|h|^2-H^2) (s f_Y+u)\,d\Sigma .
\end{aligned}
\]

We claim that
\[
\int_\Sigma (n|h|^2-H^2)(sf_Y+u)\,d\Sigma=0.
\]
Indeed, since
\[
J\varphi=(n|h|^2-H^2)s,
\qquad
Jf_Y=0,
\]
Green's formula and the Robin condition for \(\varphi\) yield
\[
\int_\Sigma (n|h|^2-H^2) s\, f_Y \,d\Sigma
= \int_{\partial \Sigma} (f_Y \nabla_\mu \varphi - \varphi \nabla_\mu f_Y) dS.
\]
Since \(\nabla_\mu \varphi=q\varphi\) on \(\partial \Sigma\), this becomes
\[
\int_\Sigma (n|h|^2-H^2) s\, f_Y \,d\Sigma
=
\int_{\partial\Sigma}
\bigl(q  f_Y-\nabla_\mu f_Y\bigr)\varphi\,dS .
\]

We now compute \(q  f_Y-\nabla_\mu f_Y\).
Let \(b_Y:=\langle Y,N\rangle\). For the Killing field \(Y=K_{a,b}=\mathcal K_{A,B}\), the boundary identity
\(
\langle \mathcal K_{A,B},N\rangle=\frac{\psi_{A,B}}{\lambda_A}
\)
gives
\(
b_Y=\frac{\psi}{\lambda}.
\)
In the horospherical case, \(\lambda=\tau\), \(\psi=-s\).
Since \(s=\tau\) on \(\mathcal P_{a,\tau}\), we have
\[b_Y=-1 \qquad \text{on } \partial \Sigma.\] 
Moreover, \(\overline\nu\) is tangent to \(\mathcal P_{a,\tau}\), and \(s\) is constant on \(\mathcal P_{a,\tau}\).
Hence, 
\[\overline\nu(b_Y)=0.\]
Applying \eqref{eq:boundary-K-unified} to \(Y\) and using \(\kappa=1\), we obtain 
\[
q f_Y-\nabla_\mu f_Y
=
\coth\gamma-\frac{h(\mu,\mu)}{\sinh\gamma}.
\]
Equations \eqref{eq:hmm} and
\eqref{eq:general-test-zero-boundary} give
\begin{equation}\label{eq:relation_H_hatH}
h(\mu,\mu)
=
H-(n-1)\cosh\gamma-\sinh\gamma\,\widehat H,
\end{equation}
where \(\widehat H\) is the mean curvature of \(\partial\Sigma\) in the flat support
\(\mathcal P_{a,\tau}\), 
and
\(\int_{\partial\Sigma}\varphi\,dS=0\).
It follows that 
\begin{equation}\label{eq:sy}
\int_\Sigma (n|h|^2-H^2) s\, f_Y\,d\Sigma
=
\int_{\partial\Sigma}\widehat H \varphi\,dS.
\end{equation}

By \eqref{eq:general-varphi-boundary},
\[
\varphi
=
\tau\sinh\gamma
\bigl(
(H-n\cosh\gamma)\inn{Y}{\overline\nu}
+
n\sinh\gamma
\bigr)
\qquad
\text{on }\partial\Sigma.
\]
On the other hand, Lemma~\ref{lem:green-potential} gives
\(\Delta\Phi_a=(n|h|^2-H^2)u\), and hence
\begin{equation}\label{eq:u_Green}
\int_\Sigma (n|h|^2-H^2)u d\Sigma=\int_{\partial\Sigma} \nabla_\mu \Phi_a dS .
\end{equation}
Using \(s=\tau\), \(u=\tau\cosh\gamma\) on \(\partial \Sigma\) and \eqref{eq:relation_H_hatH}, one computes
\begin{align}
\nabla_\mu \Phi_a
&=\tau\sinh\gamma (nh(\mu,\mu)-H)\nonumber\\
&=\tau\sinh\gamma ((n-1)(H-n\cosh\gamma)-n\sinh\gamma\widehat H).\label{eq:nabla_Phi_a}
\end{align}
Combining \eqref{eq:sy}, \eqref{eq:u_Green}, and
\eqref{eq:nabla_Phi_a} gives
\[
\int_\Sigma (n|h|^2-H^2) (sf_Y+u)d\Sigma=\tau\sinh\gamma (H-n\cosh\gamma)\int_{\partial\Sigma}(n-1+\widehat H \inn{Y}{\overline \nu})dS.
\]
Finally, since \(\partial\Sigma\) is a closed hypersurface in the flat
support
\(\mathcal P_{a,\tau}\), the classical Minkowski formula gives
\[
\int_{\partial\Sigma}
\bigl(
n-1+\widehat H\inn{Y}{\overline\nu}
\bigr)dS
=
0.
\]

The classical Minkowski formula therefore proves the claim. Hence,
\[
Q(\varphi,\varphi)+\mathcal G_a
=
n\int_\Sigma
(n|h|^2-H^2)
(u^2-s^2)
d\Sigma.
\]
Since
\(\inn{X_a}{X_a}=-s^2\),
we have
\(
u^2-s^2
=
\inn{X_a^\top}{X_a^\top}.\)
Therefore,
\[
Q(\varphi,\varphi)+\mathcal G_a
=
\int_\Sigma
(n|h|^2-H^2)\,
n\inn{X_a^\top}{X_a^\top}\,d\Sigma.
\]

\paragraph{\textbf{Flat Minkowski hyperplane support.}}
Let 
\[\mathcal P=\{x\in \mathbb{R}^{n,1}| \inn{x}{N}=0 \}, \qquad 
\inn{N}{N}=-1,\]
where \(N\) is future-directed.
The canonical test function is
\[
\varphi_0
=
n+H\langle x,\nu\rangle+n\cosh\gamma\,\langle N,\nu\rangle.
\]
It satisfies
\[J\varphi_0=-(n|h|^2-H^2), \qquad 
\nabla_\mu \varphi_0=\coth\gamma h(\mu,\mu)\varphi_0.\]
Therefore,
\[Q(\varphi_0,\varphi_0)=-\int_\Sigma (n|h|^2-H^2) \varphi_0 d\Sigma .\]
We also use the standard Green cancellation
\[
\int_\Sigma (n|h|^2-H^2)\langle x,\nu\rangle\,d\Sigma=0.
\]
Indeed, \(\langle x,\nu\rangle\) and \(\varphi_0\) satisfy the same homogeneous
Robin boundary condition, while
\[
J\langle x,\nu\rangle=H,
\qquad
J\varphi_0=-(n|h|^2-H^2),
\qquad
\int_\Sigma\varphi_0\,d\Sigma=0.
\]
Green's formula gives the claimed cancellation. Hence,
\[
\begin{aligned}
Q(\varphi_0,\varphi_0)
&=
-\int_\Sigma
(n|h|^2-H^2)
\bigl(n+H\langle x,\nu\rangle
+n\cosh\gamma\,\langle N,\nu\rangle\bigr)
\,d\Sigma  \\
&=
\int_\Sigma
(n|h|^2-H^2)
\bigl[-n(1+\cosh\gamma\,\langle N,\nu\rangle)\bigr]
\,d\Sigma .
\end{aligned}
\]
This proves the asserted identity in the flat case.
\end{proof}

\begin{remark}
The Dirichlet Green term \(\mathcal G_a\) enters both de~Sitter finite-trace
identities. In the round case it combines with the positive-definite trace to
produce the final geometric weight. In the horospherical case it additionally
supplies the contribution lost in the null degeneration of the parameter
space. Thus, the auxiliary function is not an ad hoc correction; it is the
Green potential associated with the closed conformal field \(X_a\).
\end{remark}

\section{The instability criterion and rigidity}
\label{sec:rigidity}

\subsection{Proof of the main theorem}

\begin{theorem}[Finite-dimensional instability criterion]\label{thm:criterion}
Let \(\Sigma\) satisfy all the geometric assumptions of
Theorem~\ref{thm:intro-criterion}, except stability. If \(\Sigma\) is
not totally umbilical, then
there exists an admissible mean-zero test function \(\varphi\) such that
\[
Q(\varphi,\varphi)>0.
\]
More precisely,
\begin{enumerate}[label=\textup{(\roman*)}]
\item in the round case,
\[
   \max_{1\le\alpha\le n+1}
   Q(\varphi_\alpha,\varphi_\alpha)
   \ge
   \frac1{n+1}\int_\Sigma (n|h|^2-H^2)F\,d\Sigma>0;
\]
\item in the horospherical case,
\[
   Q(\varphi,\varphi)
   \ge\int_\Sigma (n|h|^2-H^2)F\,d\Sigma>0;
\]
\item in the flat case,
\[
Q(\varphi_0,\varphi_0)=\int_\Sigma (n|h|^2-H^2)F\,d\Sigma>0.
\]
\end{enumerate}
\end{theorem}

\begin{proof}
Since the Green correction \(\mathcal G_a\) is nonpositive in the two
de~Sitter cases and no Green correction is needed in the flat case,
Theorem~\ref{thm:finite-trace} implies
\[
\sum Q(\varphi,\varphi)\geq
\int_\Sigma (n|h|^2-H^2)F\,d\Sigma.
\]
It remains to show that the right-hand side is strictly positive whenever
\(n|h|^2-H^2\not\equiv0\).

In the flat Minkowski case, \(N\) and \(\nu\) are future-directed unit
timelike fields. The reverse Cauchy--Schwarz inequality gives
\[
\inn{N}{\nu}\leq-1
\qquad\text{on }\Sigma.
\]
Hence,
\[
\begin{aligned}
F&=-n\bigl(1+\cosh\gamma\,\inn{N}{\nu}\bigr)\\
&\geq n(\cosh \gamma -1) > 0.
\end{aligned}
\]
Thus, \(n|h|^2-H^2\not\equiv 0\) implies
\(\int_\Sigma (n|h|^2-H^2)F\,d\Sigma>0\).

In the two de~Sitter cases,
\[
F=n\inn{X_a^\top}{X_a^\top}\geq0.
\]
Suppose that \(n|h|^2-H^2>0\) on a nonempty open set \(U\) and that
\(F=0\) on \(U\).
Then \(X_a^\top=0\) on \(U\), so \(X_a=\rho \nu\) for a nonvanishing function \(\rho\). 
Since \(\overline\nabla_VX_a=-\inn{x}{a}V\), for tangent vector fields
\(V,W\) on \(U\) we obtain
\[
-\inn{x}{a}\inn{V}{W}=
\inn{\overline\nabla_V X_a}{W} =\inn{\rho \overline\nabla_V\nu}{W}= \rho h(V,W).
\]
Thus, \(\Sigma\) is umbilical on \(U\), contradicting
\(n|h|^2-H^2>0\) there. Consequently,
\((n|h|^2-H^2)F\) is positive somewhere whenever
\(n|h|^2-H^2\not\equiv0\), and hence its integral is strictly positive.
\end{proof}

\begin{proof}[Proof of Theorem~\ref{thm:intro-criterion}]
If \(\Sigma\) were not totally umbilical,
Theorem~\ref{thm:criterion} would provide an admissible mean-zero test
function with positive second variation, contradicting stability.
\end{proof}

\subsection{Stability of totally umbilical caps}

Let \(\mathcal P_A\subset\mathbb Q_1^{n+1}(c)\) be the totally umbilical
support introduced above. Suppose that
\(\Sigma\subset\mathbb Q_1^{n+1}(c)\) is compact, connected, and totally
umbilical, with nonempty boundary
\(\partial\Sigma\subset\mathcal P_A\). By the classification of totally
umbilical hypersurfaces in Lorentzian space forms~\cite{Sato21},
\(\Sigma\) is capillary and isometric to a closed geodesic ball \(B_R^K\)
in the \(n\)-dimensional Riemannian space form of curvature \(K\).
We call such \(\Sigma\) a totally umbilical cap.
The main theorem establishes that every stable capillary hypersurface is
a totally umbilical cap. The following proposition proves the converse.

\begin{proposition}
Let
\(\Sigma^n\subset\mathbb{Q}^{n+1}_1(c)\)
be a compact connected spacelike totally umbilical capillary
hypersurface whose boundary lies on a spacelike totally umbilical
support \(\mathcal{P}_A\) and meets \(\mathcal{P}_A\) at a constant angle
\(\gamma>0\). Then \(\Sigma\) is stable.
\end{proposition}

\begin{proof}
Since both \(\Sigma\) and \(\mathcal P_A\) are totally umbilical,
we may write
\[
    h=\sigma g,
    \qquad
    h^{\mathcal P_A}
    =\kappa_A\,\overline g|_{T\mathcal P_A}.
\]
By the Codazzi equation, \(\sigma\) is constant. The Lorentzian
Gauss equation therefore gives
\[
    K=\operatorname{Sec}(\Sigma)=c-\sigma^2.
\]
Since \(|h|^2=n\sigma^2\), the Jacobi operator reduces to
\[
    J=\Delta-|h|^2+nc
      =\Delta+nK.
\]

By Proposition~\ref{prop:boundary-relations}, the second fundamental
form of \(\partial\Sigma\subset\Sigma\), with respect to the outward
unit conormal \(\mu\), is
\[
    \widetilde h
    =
    \left(
        \sigma\coth\gamma
        -\kappa_A\operatorname{csch}\gamma
    \right)g|_{T\partial\Sigma}.
\]
Thus, the Robin coefficient
\[
    q_A
    =
    \sigma\coth\gamma
    -\kappa_A\operatorname{csch}\gamma
\]
is precisely the intrinsic principal curvature of
\(\partial\Sigma\subset\Sigma\). We henceforth write \(q=q_A\).

Under the intrinsic identification
\[
    (\Sigma,g)\cong B_R^K,
\]
one has
\[
    q
    =
    \frac{\operatorname{sn}_K'(R)}
         {\operatorname{sn}_K(R)},
\]
where \(\operatorname{sn}_K\) denotes the generalized sine function of
the space form of curvature \(K\).
Using the second variation formula and integrating by parts, we obtain
\begin{equation}\label{eq:Q-intrinsic}
    -Q(f,f)
    =
    \int_\Sigma
    \left(
        |\nabla f|^2-nKf^2
    \right)d\Sigma
    -
    \frac{\operatorname{sn}_K'(R)}
         {\operatorname{sn}_K(R)}
    \int_{\partial\Sigma}f^2\,dS.
\end{equation}
Thus, \(-Q\) is the intrinsic Robin index form on the geodesic ball
\(B_R^K\) associated with
\[
\begin{cases}
    -\Delta u-nKu=\lambda u,
        &\text{in }B_R^K,\\[1mm]
    \displaystyle
    \partial_\mu u
    -\frac{\operatorname{sn}_K'(R)}
           {\operatorname{sn}_K(R)}u=0,
        &\text{on }\partial B_R^K.
\end{cases}
\]

The nonnegativity of this index form on the mean-zero space
\[
    H^1_*(B_R^K)
    :=
    \left\{
        f\in H^1(B_R^K):
        \int_{B_R^K}f\,d\Sigma=0
    \right\}
\]
is precisely the stability result proved in
\cite[Proposition~4.3 and Appendix~B]{DamascenoElbert24}.
Consequently,
\[
    -Q(f,f)\geq0
\]
for every \(f\in C^\infty(\Sigma)\) satisfying
\(\int_\Sigma f\,d\Sigma=0\), and hence \(\Sigma\) is stable.
\end{proof}

\section{Negative-curvature supports and the conformal interpretation}
\label{sec:negative-curvature}

The light-cone construction and the Jacobi--Robin identities do not require
the support to have nonnegative intrinsic curvature. Hence, they remain
valid for hyperbolic slices in de~Sitter space,
hyperboloid supports in Minkowski space, and spacelike totally
umbilical supports in anti-de~Sitter space. In these cases, however,
the finite-dimensional trace is no longer positive. We first describe
the relevant model supports and then identify the resulting signature
obstruction.

\subsection{Negative-curvature model supports}
\label{subsec:negative-supports}

\paragraph{\textbf{Hyperbolic slices in de~Sitter space}}

Let \(\mathcal P_{a,\tau}\) be a connected component of
\[
\bigl\{
x\in\mathbb S^{n+1}_1 |
\langle x,a\rangle=\tau
\bigr\},
\qquad
\langle a,a\rangle=1,
\qquad
\tau^2>1.
\]
Then
\[
\lambda_{a,\tau}=\sqrt{\tau^2-1},
\qquad
\operatorname{Sec}(\mathcal P_{a,\tau})=-\lambda_{a,\tau}^{-2}<0.
\]
For \(b\in a^\perp\), the construction of Section~\ref{sec:light_cone} gives
\[
\psi_{a,b}=\langle x,b\rangle,
\]
and the associated Jacobi--Robin test function is
\[
\varphi_{a,b,\tau}
=
n\psi_{a,b}
+
H\langle T_{a,b,\tau},\nu\rangle
-
n\lambda_{a,\tau}\cosh\gamma\langle K_{a,b},\nu\rangle.
\]
The corresponding Minkowski formula is
\[
\int_\Sigma \varphi_{a,b,\tau}\,d\Sigma=0.
\]
Since \(a\) is spacelike, \(a^\perp\) has Lorentzian signature
\((n,1)\). Thus, the associated parameter space contains one timelike
direction.

\paragraph{\textbf{Hyperboloid supports in Minkowski space}}

We next record the formula for a spacelike hyperboloid support in
\(\mathbb R^{n,1}\). Since translations of
\(\mathbb R^{n,1}\) are ambient isometries, we may assume that the
hyperboloid is centered at the origin.
Let \(\mathcal P_r\) be one connected component of
\[
\{x\in\mathbb R^{n,1}:\langle x,x\rangle=-r^2\}
\qquad r>0
\]
for which \(x/r\) is future-directed, and choose  the unit timelike normal
\[
N=\frac{x}{r}.
\]
Then \(\mathcal{P}_r\) is a spacelike totally umbilical hypersurface with principal curvature \(r^{-1}\)
and intrinsic sectional curvature
\[
\operatorname{Sec}(\mathcal P_r)=-\frac1{r^2}.
\]

In the flat light-cone model, \(\mathcal P_r=\mathcal P_A\) for
\[A=E_0^*-\frac{r^2}{2} E_0.\]
For \(k\in \mathbb{R}^{n,1}\), set \(B=k\). The associated fields are
\[\mathcal{K}_{k}= k, \qquad  
\mathcal{T}_{k}=\langle x,k\rangle  x+\frac{r^2-\langle x,x\rangle}{2}k, \qquad
\psi_{k}=\langle x,k\rangle.\]
Hence, every compact spacelike capillary hypersurface supported on \(\mathcal P_r\) satisfies
\[
\int_{\Sigma} \left[n\langle x,k\rangle +H \left\langle \langle x,k\rangle  x
+\frac{r^2-\langle x,x\rangle}{2} k,\nu 
\right\rangle -nr\cosh \gamma \langle k,\nu\rangle \right] d\Sigma=0.
\]
A hyperboloid centered at \(p\in\mathbb R^{n,1}\) is obtained by
replacing \(x\) with \(x-p\).

\paragraph{\textbf{Spacelike totally umbilical supports in anti-de~Sitter
space.}}

Let 
\[
\operatorname{AdS}^{n+1}
=
\bigl\{
x\in\mathbb R^{n,2}:
\langle x,x\rangle=-1
\bigr\}
\]
and let \(\mathcal{P}_{a,\tau}\) be one connected component of 
\[
\big\{x\in \operatorname{AdS}^{n+1}|\langle x,a\rangle=\tau\big\}.
\]
The support is spacelike precisely when
\[
\tau^2<-\langle a,a\rangle.
\]
Set
\[\lambda_{a,\tau}=\sqrt{-\langle a,a\rangle-\tau^2}.\]
With the convention induced by the light-cone model, its unit timelike
normal and principal curvature are
\[
N_{a,\tau}
=
-\frac{a+\tau x}{\lambda_{a,\tau}}, \qquad 
\kappa_{a,\tau}
=
-\frac{\tau}{\lambda_{a,\tau}}.
\]
Consequently,
\[
\operatorname{Sec}(\mathcal{P}_{a,\tau})
=
-1-\kappa_{a,\tau}^2
=\frac{ \langle a, a \rangle}{\lambda_{a,\tau}^2}<0.
\]

For \(b\in \mathbb{R}^{n,2}\), put
\[
X_a:
=
a+\langle x,a\rangle x,
\qquad
X_b:
=
b+\langle x,b\rangle  x,
\]
\[
K_{a,b}:
=-\langle x,a\rangle b+\langle x,b\rangle  a,
\]
and
\[
    T_{a,b,\tau}\coloneqq \mathcal T_{A,B}
    =\langle a,b\rangle X_a-\langle a,a\rangle X_b+\tau K_{a,b}.
\]
The corresponding conformal factor is 
\[
\psi_{a,b}
=
-\langle a\wedge x,a\wedge b\rangle.
\]
Thus, the unified Minkowski-type formula \eqref{eq:Minkowski-general} gives
\[
\int_\Sigma
\left(
-n\langle a\wedge x,a\wedge b\rangle
+
H\langle T_{a,b,\tau},\nu\rangle
-
n\lambda_{a,\tau}\cosh\gamma\,
\langle K_{a,b},\nu\rangle
\right)d\Sigma
=
0.
\]

\begin{remark}
When \(n\ge2\), every connected spacelike totally umbilical
hypersurface in \(\operatorname{AdS}^{n+1}\) is, up to restriction to
an open subset and reversal of the unit normal, of this form. In
particular, all such supports have negative intrinsic sectional
curvature.
\end{remark}

\subsection{The signature obstruction}
\label{subsec:signature-obstruction}

Let
\[
\Pi_A=\operatorname{span}\{A,E_c\}
\]
be the support two-plane. If \(\operatorname{Sec}(\mathcal P_A)<0\), then 
Proposition~\ref{prop:test-signature} gives
\[
\operatorname{sign}(\Pi_A)=(1,1),
\qquad
\operatorname{sign}(\Pi_A^\perp)=(n,1).
\]
Since the test construction vanishes on \(\Pi_A\), the natural
parameter space is
\[
\mathbb V/\Pi_A\simeq\Pi_A^\perp.
\]
It therefore contains exactly one timelike direction.

This is precisely the obstruction encountered by the finite-trace method.
The Minkowski-type formula, the mean-zero condition, and the Jacobi--Robin
identities remain valid. However, the invariant trace over
\(\Pi_A^\perp\) has one term with the opposite sign, and stability
alone does not control this contribution. Thus, the essential issue is
not the construction of additional test functions but the control of
the unique timelike mode.

\begin{problem}\label{prob:negative-curvature}
For a spacelike totally umbilical support with
\(\operatorname{Sec}(\mathcal P_A)<0\), determine whether the timelike
direction in the finite Jacobi--Robin test module can be controlled by an
additional Green identity, a geometric sign condition, or a spectral
constraint. The principal cases are:
\begin{enumerate}[label=\textup{(\roman*)}]
\item hyperbolic-slice supports in de~Sitter space;
\item hyperboloid supports in Minkowski space;
\item spacelike totally umbilical supports in anti-de~Sitter space.
\end{enumerate}
\end{problem}

\subsection{Conformal interpretation and the Riemannian real form}
\label{subsec:conformal-interpretation}

We conclude by recording the conformal interpretation of the preceding
construction. We use the following terminology only as a compact
reformulation of the ambient-vector construction; no tractor calculus is
required in the proofs above.

In the flat conformal model, \(\mathbb V\) is the standard tractor vector
space, and a constant vector \(A\in\mathbb V\) represents a parallel standard
tractor. Its scale component in the chosen space-form metric is
\[
s_A(\xi)=\langle \xi,A\rangle_{\mathbb V},
\]
which, in the chosen space-form scale, satisfies
\[
\overline\nabla^2 s_A=-\rho_A\overline g,
\qquad
d\rho_A=c\,ds_A.
\]
In this subsection, an almost Einstein scale means a function \(s\) for which
there exists a function \(\rho\) satisfying
\[
\overline\nabla^2s=-\rho\,\overline g,
\qquad
d\rho=c\,ds.
\]
Thus, \(s_A\) is both a special concircular potential and an almost Einstein
scale in this sense. The regular zero set
\[
\mathcal P_A=\{s_A=0\}
\]
is what we call the conformal hypersphere associated with \(A\). In the chosen
space-form metric it is a totally umbilical hypersurface. The fields \(\mathcal K_{A,B}\) and
\(\mathcal T_{A,B}\) arise from the algebra generated by
parallel tractors \(A\) and \(B\). In particular, the test functions used above
are canonically associated with a finite-dimensional tractor space.

The same construction admits an intrinsic metric formulation. 
Let
\((\overline M^{n+1},\overline g)\) be an Einstein manifold with
\[
\operatorname{Ric}_{\overline g}=nc\,\overline g,
\]
and let \(s,t\in C^\infty(\overline M)\) be special concircular
potentials satisfying
\[
\overline\nabla^2s=-\rho_s\overline g,\qquad d\rho_s=c\,ds,
\]
\[
\overline\nabla^2t=-\rho_t\overline g,\qquad d\rho_t=c\,dt.
\]
Thus,
\[
\rho_s=m_s+cs,\qquad \rho_t=m_t+ct
\]
for constants \(m_s,m_t\). Define a pairing
\[
I(s,t)
=
\overline g(\overline\nabla s,\overline\nabla t)
+
\rho_s t+\rho_t s-cst.
\]
This pairing is constant by the defining equations for \(s\) and \(t\).
Assume that the support
\[
\mathcal P_s=\{s=0\},
\]
is a regular hypersurface and \(\overline\nabla s\) is timelike along \(\mathcal P_s\).
Set
\[
\lambda_s^2=-I(s,s)>0,
\qquad
N=-\lambda_s^{-1}\overline\nabla s.
\]
Then \(\mathcal P_s\) is spacelike and totally umbilical, with principal
curvature
\(
\kappa_s={m_s}/{\lambda_s}.
\)

For such \(s,t\), let
\[
X_s:=\overline\nabla s,\qquad X_t:=\overline\nabla t, \qquad \mathfrak C_{s,t}:=sX_t-tX_s,
\]
and define
\begin{align*}
&\mathcal K_{s,t}:=\rho_sX_t-\rho_tX_s,\\
&\mathcal T_{s,t}:=I(s,t)X_s-I(s,s)X_t + m_s\mathfrak C_{s,t},\\
&\psi_{s,t} := I(s,s)\rho_t-I(s,t)\rho_s + m_s(t\rho_s-s\rho_t).
\end{align*}
A direct computation gives that \(\mathcal K_{s,t}\) is Killing and
\(\mathcal T_{s,t}\) is conformal Killing with conformal factor
\(\psi_{s,t}\):
\[
\mathcal L_{\mathcal T_{s,t}}\overline g=2\psi_{s,t}\overline g.
\]
Moreover,
\[
\mathcal T_{s,t}=\lambda_s^2(X_t)^{\mathcal P_s}
\quad\text{on }\mathcal P_s,
\]
and hence \(\mathcal T_{s,t}\) is tangent to the support.

Consequently, if
\(\Sigma^n\subset \overline M^{n+1}\) is a compact spacelike capillary
hypersurface with
\[
\partial\Sigma\subset\mathcal P_s,
\qquad
\overline g(\nu,N)=-\cosh\gamma,
\]
then the same divergence argument gives the intrinsic Minkowski-type
formula
\[
\int_\Sigma
\left(
n\psi_{s,t}
+
H\overline g(\mathcal T_{s,t},\nu)
-
n\lambda_s\cosh\gamma\,
\overline g(\mathcal K_{s,t},\nu)
\right)d\Sigma
=0.
\]
In the flat conformal model, \(s=s_A\) and \(t=s_B\) are precisely the
scale components of the parallel tractors \(A\) and \(B\), and this
formula reduces to the light-cone formula proved above.

This Minkowski-type identity alone does not imply rigidity. The stability
argument also requires a sufficiently large finite-dimensional space of
such potentials and a trace of the stability form with a favorable sign.

The Riemannian and Lorentzian constructions are two forms of the
same tractor-algebraic mechanism. In the Riemannian case, the conformal
ambient space is the light cone in \(\mathbb R^{n+2,1}\), whereas in the Lorentzian
case it is modeled in \(\mathbb R^{n+1,2}\). The formal expressions for
\(\mathcal K_{A,B}\), \(\mathcal T_{A,B}\), \(\psi_{A,B}\), \(\varphi_{A,B}\)
and the Jacobi--Robin test package are essentially the same;
the difference lies in the signature of the ambient metric.

\begin{remark}
It is instructive to compare these results with the corresponding
Riemannian theory. To the best of our knowledge, the case of totally umbilical support hypersurfaces
with negative intrinsic curvature remains open in the present Lorentzian
setting. In the Riemannian analogue, the corresponding cases are
equidistant and totally geodesic hypersurfaces in hyperbolic space,
which likewise remain open.
\end{remark}

\appendix

\section{Second variation of the capillary energy}
\label{sec:second-var-unified}

We derive the second-variation formula using the Lorentzian conventions
of this paper. In particular, the signs are verified directly rather than
transferred from a Riemannian second-variation formula.

Recall that \((\overline M^{n+1}(c),\overline g)\) is a Lorentzian space form
of constant sectional curvature \(c\). Let \(\mathcal P\subset \overline M\)
be a spacelike totally umbilical hypersurface with unit timelike normal
\(N\) and second fundamental form
\[
h^{\mathcal{P}}(X,Y)=\kappa\,\overline g(X,Y)\qquad (X,Y\in T\mathcal{P}),
\]
where \(\kappa\) is a constant.

Let \(x:\Sigma^n\to \overline M^{n+1}(c)\) be a compact spacelike immersion with boundary \(\partial\Sigma\subset\mathcal{P}\).
The unit normals and conormals are related by
\begin{gather*}
    \nu = \cosh\gamma\,N + \sinh\gamma\,\overline{\nu}  \quad \text{ and } \quad
    \mu = \sinh\gamma\,N + \cosh\gamma\,\overline{\nu}\\
     N = \cosh\gamma\,\nu - \sinh\gamma\,\mu  \quad \text{ and } \quad
    \overline{\nu}=-\sinh\gamma\,\nu+\cosh\gamma\,\mu.
\end{gather*}
In the sequel, \(\overline\nabla\) denotes the Levi--Civita connection of
\((\overline M,\overline g)\), while \(\nabla=\nabla^\Sigma\),
\(\nabla^{\mathcal P}\), and \(\nabla^\partial\) denote the induced
connections on \(\Sigma\), \(\mathcal P\), and \(\partial\Sigma\),
respectively. We write \(g=\overline g|_{T\Sigma}\) for the induced
Riemannian metric. Ambient inner products involving \(\nu\), \(N\), \(Y\),
or other vectors in \(T\overline M\) are always written with \(\overline g\).

Assume that \(H\) and \(\gamma\) are constant, so that \(\Sigma\) is a
critical point of the functional
\[
    \mathcal E(t)=\mathcal A(t)-\cosh\gamma\,\mathcal W(t).
\]
Let \(x:(-\varepsilon,\varepsilon)\times\Sigma\to \overline M\) be an admissible
variation with \(x_t(\partial\Sigma)\subset\mathcal P\), and set
\[
Y_t:=\frac{\partial x}{\partial t},
\qquad
Y:=Y_0.
\]
We use the \(t\)-dependent normal and conormal fields
\(\nu_t,\mu_t,N_t\), and \(\overline\nu_t\). Their variations are written
directly as ambient covariant derivatives, for example
\[
\left.\overline\nabla_{\frac d{dt}}\nu_t\right|_{t=0},
\qquad
\left.\overline\nabla_{\frac d{dt}}\mu_t\right|_{t=0},
\qquad
\left.\overline\nabla_{\frac d{dt}}Y_t\right|_{t=0}.
\]
Write
\[
    Y=f\nu+Y^\top,
\]
where \(f=-\overline g(Y,\nu)\) and \(Y^\top\) is tangent to \(\Sigma\).
Along \(\partial\Sigma\), using the fact that \(\overline g(Y,N)=0\),
we can further write
\[
    Y=f\nu-f\coth\gamma\,\mu+Y^{\partial\Sigma},
\]
where \(Y^{\partial\Sigma}\) is tangent to \(\partial\Sigma\).

\begin{lemma}\label{pointvariation}
    Consider the following shape operators:
\begin{itemize}
    \item \(S_0\) is the shape operator of \(\Sigma\) with respect to
        \(\nu\), defined by
        \(S_0(X)=\overline\nabla_X\nu\) for \(X\in T\Sigma\);
    \item \(S_1\) is the shape operator of \(\partial\Sigma\) in
        \(\Sigma\) with respect to the outward conormal \(\mu\), defined by
        \(S_1(X)=\nabla^\Sigma_X\mu\) for \(X\in T(\partial\Sigma)\);
    \item \(S_2\) is the shape operator of \(\partial\Sigma\) in
        \(\mathcal P\) with respect to \(\overline\nu\), defined by
        \(S_2(X)=\nabla^{\mathcal P}_X\overline\nu\) for
        \(X\in T(\partial\Sigma)\).
\end{itemize}
    Then
\begin{enumerate}
    \item  $\displaystyle\left.\overline{\nabla}_{\frac d{dt}}\nu_t\right|_{t=0} = \nabla f + S_{0}(Y^{\top}),$
    \item  $\displaystyle\left.\overline{\nabla}_{\frac d{dt}}\mu_t\right|_{t=0} = \bigl(\nabla_{\mu} f +h(Y^{\top},\mu)\bigr)\nu - f S_0(\mu) + fh(\mu,\mu)\mu + S_1(Y^{\partial\Sigma}) + \coth\gamma\, \nabla^{\partial} f,$
    \item  $\displaystyle\left.\overline{\nabla}_{\frac d{dt}}\overline\nu_t\right|_{t=0} = h^{\mathcal{P}}(Y,\overline\nu) N + S_{2}(Y^{\partial\Sigma}) + \frac{1}{\sinh\gamma}  \nabla^{\partial} f.$
\end{enumerate}
\end{lemma}

\begin{proof}
    To prove (1), let \(X\in T\Sigma\) be arbitrary. Differentiating the
    identity
    \[
        \overline g(\nu,X)=0
    \]
    with respect to \(t\) yields
    \[
        \overline g\left(\overline{\nabla}_{\frac d{dt}}\nu_t,X\right)
        +\overline g\left(\nu,\overline{\nabla}_{\frac d{dt}}X\right)=0.
    \]
    Thus,
    \begin{align*}
        \overline g\left(\overline{\nabla}_{\frac d{dt}}\nu_t,X\right)
        &=-\overline g\left(\nu,\overline{\nabla}_{\frac d{dt}}X\right)
        =-\overline g(\nu,\overline{\nabla}_XY) \\
        &=-\overline g(\nu,\overline{\nabla}_X(f\nu+Y^\top))
        = Xf+g(\overline{\nabla}_X\nu, Y^\top) \\
        & = g(\nabla f,X)+g(S_0(Y^\top),X).
    \end{align*}
    Moreover, \(\overline g(\nu,\nu)=-1\) gives
    \[
    \overline g\left(\overline{\nabla}_{\frac d{dt}}\nu_t,\nu\right)=0.
    \]
    This proves (1).

    Arguing as in the proof of (1), we obtain
    \begin{align*}
        \overline g\left(\overline{\nabla}_{\frac d{dt}}\mu_t,X\right)
        &=-\overline g\left(\mu,\overline{\nabla}_X(f\nu-f\coth\gamma\,\mu+Y^{\partial\Sigma})\right) \\
        &=-f\overline g(\mu,\overline{\nabla}_X\nu)+\coth\gamma\, g(\nabla^\partial f,X)+g(S_1(Y^{\partial\Sigma}),X) \\
        & =g(\coth\gamma\nabla^\partial f+S_1(Y^{\partial\Sigma})-fS_0(\mu),X),
    \end{align*}
    for \(X\in T\partial\Sigma\),
    \begin{align*}
        \overline g\left(\overline{\nabla}_{\frac d{dt}}\mu_t,\nu\right)
        &=-\overline g\left(\mu,\overline{\nabla}_{\frac d{dt}}\nu_t\right) \\
        &=-g(\mu,\nabla f+S_0(Y^\top)) \\
        & =-\mu(f)-h(Y^\top,\mu),
    \end{align*}
    and
    \(\overline g(\overline{\nabla}_{\frac d{dt}}\mu_t,\mu)=0\).
    Thus,
    \[
        \left.\overline{\nabla}_{\frac d{dt}}\mu_t\right|_{t=0}
        =\bigl(\nabla_{\mu} f +h(Y^{\top},\mu)\bigr)\nu
        -fS_0(\mu)+fh(\mu,\mu)\mu
        +S_1(Y^{\partial\Sigma})+\coth\gamma\nabla^\partial f.
    \]

    For (3),
    \begin{align*}
        \overline g\left(\overline{\nabla}_{\frac d{dt}}\overline\nu_t,X\right)
        &=-\overline g\parens*{\overline{\nu},\overline{\nabla}_X\parens*{-\frac{f}{\sinh\gamma}\overline{\nu}+Y^{\partial\Sigma}}} \\
        & =g\parens*{S_2(Y^{\partial\Sigma})+\frac{1}{\sinh\gamma}\nabla^\partial f,X},
    \end{align*}
    for \(X\in T\partial\Sigma\),
    \[
        \overline g\left(\overline{\nabla}_{\frac d{dt}}\overline\nu_t,N\right)
        =-\overline g(\overline\nu,\overline\nabla_YN)
        =-h^{\mathcal P}(Y,\overline\nu)
        \quad \text{and} \quad
        \overline g\left(\overline{\nabla}_{\frac d{dt}}\overline\nu_t,\overline\nu\right)=0.
    \]
    Thus,
    \[
        \left.\overline{\nabla}_{\frac d{dt}}\overline\nu_t\right|_{t=0}
        =h^{\mathcal P}(Y,\overline\nu)N+S_{2}(Y^{\partial\Sigma})
        +\frac{1}{\sinh\gamma}\nabla^{\partial}f. \qedhere
    \]
\end{proof}

\begin{theorem}[Second variation of the capillary energy]
\label{thm:second-var-unified}
Let \(x:\Sigma\to \overline M\) be a capillary immersion, and let
\(x_t\) be a volume-preserving admissible variation. Then
\begin{equation}\label{eq:second-var-unified}
\mathcal E''(0)
=\int_{\Sigma}f\,Jf\,d\Sigma
-\int_{\partial\Sigma}f\bigl(\nabla_{\mu}f-qf\bigr)\,dS,
\end{equation}
where \(J=\Delta-|h|^{2}-\overline{\operatorname{Ric}}(\nu,\nu)\) is the Jacobi operator of \(\Sigma\) and
\[
q=\coth\gamma\,h(\mu,\mu)
-\frac{1}{\sinh\gamma}\,h^{\mathcal P}(\overline\nu,\overline\nu).
\]

\end{theorem}

\begin{proof}
From the first variation formula (Lemma~\ref{lem:first-var-unified}), we have
\[
    \mathcal E'(t)=\int_{\Sigma_t} -H(t)\,\overline g(Y_t,\nu_t)\,d\Sigma_t
    +\int_{\partial\Sigma_t} \overline g\bigl(Y_t,\;\mu_t-\cosh\gamma\,\overline\nu_t\bigr)\,dS_t .
\]
Differentiating again, we obtain
\begin{align*}
\mathcal{E}''(0)=& \int_{\Sigma} -H'(0)\,\overline g(Y,\nu)\,d\Sigma
    -H(0)\,\frac{d}{dt}\Big|_{t=0}\!\left(\int_{\Sigma_t}\overline g(Y_t,\nu_t)\,d\Sigma_t\right) \\
    & +\int_{\partial\Sigma}
    \overline g\!\left(
    \left.\overline\nabla_{\frac d{dt}}Y_t\right|_{t=0},
    \mu-\cosh\gamma\,\overline\nu\right)dS \\
    & +\int_{\partial\Sigma}
    \overline g\!\left(
    Y,
    \left.\overline\nabla_{\frac d{dt}}
    \bigl(\mu_t-\cosh\gamma\,\overline\nu_t\bigr)\right|_{t=0}
    \right)dS \\
    & +\int_{\partial\Sigma} \overline g(Y,\mu-\cosh\gamma\,\overline\nu)\,
    \frac{d}{dt}\Big|_{t=0}(dS_t) .
\end{align*}
For a volume-preserving variation,
\[
\frac{d}{dt}\Big|_{t=0}\!\left(\int_{\Sigma_t}\overline g(Y_t,\nu_t)\,d\Sigma_t\right) = -\mathcal V''(0) = 0,
\]
and $\overline g(Y,\mu-\cosh\gamma\,\overline\nu)=0$, because $Y$ is tangent to $\mathcal{P}$ along $\partial\Sigma$.
By the Lorentzian mean-curvature variation formula of
Barbosa--Oliker~\cite{BO93}, with the conventions used here and with
\(H\) constant at \(t=0\),
\begin{equation}\label{varofH}
    \left.\frac{d}{dt}\right|_{t=0}H(t)
    =\Delta f-|h|^{2}f-\overline{\operatorname{Ric}}(\nu,\nu)f
    =Jf.
\end{equation}

Since $\overline g(Y,\nu) = -f$, the first integral equals $\int_{\Sigma} f\,Jf\,d\Sigma$.

To handle the boundary terms, we need
\[
\mathcal{B}:= \overline g\!\left(\left.\overline\nabla_{\frac d{dt}}Y_t\right|_{t=0},\; \mu-\cosh\gamma\,\overline\nu\right)
        + \overline g\!\left(Y,\; \left.\overline\nabla_{\frac d{dt}}\bigl(\mu_t-\cosh\gamma\,\overline\nu_t\bigr)\right|_{t=0}\right) .
\]

Noting that $\mu - \cosh\gamma\,\overline\nu = \sinh\gamma\,N$ and $\overline g(Y,N)=0$
along $\partial\Sigma$, we compute
\begin{equation}\label{vterm2}
    \overline g\!\left(\left.\overline\nabla_{\frac d{dt}}Y_t\right|_{t=0},\; \mu-\cosh\gamma\,\overline\nu\right)
= \sinh\gamma\,\overline g\!\left(\left.\overline\nabla_{\frac d{dt}}Y_t\right|_{t=0},N\right)
= -\sinh\gamma\, h^{\mathcal P}(Y,Y). 
\end{equation}

For the second term, we use parts (2) and (3) of
Lemma~\ref{pointvariation}:
\begin{align}\label{vterm1}
\overline g\!\left(Y,
\left.\overline\nabla_{\frac d{dt}}\mu_t\right|_{t=0}
-\cosh\gamma\left.\overline\nabla_{\frac d{dt}}\overline\nu_t\right|_{t=0}\right)
&=- f\,\nabla_{\mu}f + f^2 \coth\gamma\,h(\mu,\mu)
+ \sinh\gamma\,h^{\mathcal P}(Y^{\partial\Sigma},Y^{\partial\Sigma}) .
\end{align}
Here we have used the boundary decomposition
\[
Y = f\nu + Y^{\top} = f\nu + Y^{\partial\Sigma} - \coth\gamma\,f\mu 
\]
and
\[
g\bigl(Y^{\partial\Sigma},S_{1}(Y^{\partial\Sigma})\bigr)
-\cosh\gamma\,g\bigl(Y^{\partial\Sigma},S_{2}(Y^{\partial\Sigma})\bigr)
= \sinh\gamma\,h^{\mathcal P}(Y^{\partial\Sigma},Y^{\partial\Sigma}) .
\]
We have also used the principal-direction property
$h(Y^{\partial\Sigma},\mu)=0$, which follows from the total
umbilicity of the support and the constant-angle condition.

Now expand $h^{\mathcal P}(Y,Y)$ using the alternative decomposition
$Y = Y^{\partial\Sigma} - \frac{1}{\sinh\gamma} f \overline\nu$:
since the support is totally umbilical,
\[
h^{\mathcal P}(Y^{\partial\Sigma},\overline\nu)
=\kappa\,\overline g(Y^{\partial\Sigma},\overline\nu)=0.
\]
\begin{equation}\label{hdecomp}
    h^{\mathcal P}(Y,Y) = h^{\mathcal P}(Y^{\partial\Sigma},Y^{\partial\Sigma}) + \frac{1}{\sinh^2\gamma} f^2 h^{\mathcal P}(\overline\nu,\overline\nu). 
\end{equation}

Finally, combining \eqref{vterm2}, \eqref{vterm1}, and
\eqref{hdecomp}, we obtain
\begin{align*}
\mathcal{B}
&=\overline g\!\left(\left.\overline\nabla_{\frac d{dt}}Y_t\right|_{t=0},
\mu-\cosh\gamma\,\overline\nu\right)\\
&\quad+\overline g\!\left(Y,
\left.\overline\nabla_{\frac d{dt}}\mu_t\right|_{t=0}
-\cosh\gamma\left.\overline\nabla_{\frac d{dt}}\overline\nu_t\right|_{t=0}\right)\\
&= -f\,\nabla_{\mu}f + \Bigl(\coth\gamma\,h(\mu,\mu)
- \frac{1}{\sinh\gamma}\,h^{\mathcal P}(\overline\nu,\overline\nu)\Bigr) f^{2}\\
&=-f(\nabla_\mu f-qf).
\end{align*}
Combining the interior and boundary contributions, we obtain
\[
\mathcal E''(0)=\int_\Sigma fJf\,d\Sigma
-\int_{\partial\Sigma}f(\nabla_\mu f-qf)\,dS.
\]
This proves \eqref{eq:second-var-unified}.
\end{proof}

\begin{ack}
The authors were partially supported by the National Natural Science
Foundation of China (grant nos.~12471048 and W2521103) and the China
Postdoctoral Science Foundation (grant no.~2025M773111).
\end{ack}


\begin{thebibliography}{99}

 \bibitem{AinouzSouam16}
A.~Ainouz and R.~Souam.
\newblock{Stable capillary hypersurfaces in a half-space or a slab.}
\newblock{Indiana Univ. Math. J.}, {\bf 65} (2016), 813--831.

\bibitem{Akutagawa87}
K. Akutagawa.
\newblock{On spacelike hypersurfaces with constant mean curvature in the de Sitter space.}
\newblock{Math. Z.}, {\bf 196} (1987), 13--19.

\bibitem{BCE}
J.~L.~M. Barbosa, M.~P. do~Carmo and J.-H. Eschenburg.
\newblock{Stability of hypersurfaces of constant mean curvature in Riemannian manifolds.}
\newblock{Math. Z.}, {\bf 197} (1988), no.~1, 123--138.

\bibitem{BO93}
J.~L.~M. Barbosa and V.~I. Oliker.
\newblock{Spacelike hypersurfaces with constant mean curvature in Lorentz space.}
\newblock{Mat. Contemp.}, {\bf 4} (1993), 27--44.

\bibitem{ChenDengXieYin25}
J.~Chen, H.~Deng, H.~Xie and J.~Yin.
\newblock{Stability and rigidity results of space-like hypersurface in the Minkowski space.}
\newblock{arXiv:2506.01012v3, 2025.}



\bibitem{CY}
S. Y. Cheng and S. T. Yau.
\newblock{Maximal Space-like Hypersurfaces in the Lorentz-Minkowski Spaces.}
\newblock{Ann. of Math.}, {\bf 104} (1976), 407--419.


\bibitem{CK15}
J. Choe and M. Koiso.
\newblock{Stable capillary hypersurfaces in a wedge.}
\newblock{Pac. J. Math.}, {\bf 280} (2015), 1--15.

\bibitem{ConcusFinn}
P. Concus and R. Finn.
\newblock{On the Behavior of a Capillary Surface in a Wedge.}
\newblock{Proc. Natl. Acad. Sci.}, {\bf 63} (1969), 292--299.

\bibitem{DamascenoElbert24} 
L. Damasceno and M.~F. Elbert. 
\newblock{Stability of capillary hypersurfaces with constant higher order mean curvature.} 
\newblock{J. Geom. Anal.}, {\bf 34} (2024), no.~12, 377.

\bibitem{Finn}
R. Finn.
\newblock{Equilibrium capillary surfaces.}
\newblock{Springer-Verlag, New York}, (1986).

\bibitem{GuoWangXia22}
J.~Guo, G.~Wang and C.~Xia.
\newblock{Stable capillary hypersurfaces supported on a horosphere in the hyperbolic space.}
\newblock{Adv. Math.}, {\bf 409} (2022), Paper No.~108641, 25~pp.  

\bibitem{LiXiong17}
H.~Li and C.~Xiong.
\newblock{Stability of capillary hypersurfaces with planar boundaries.}
\newblock{J. Geom. Anal.}, {\bf 27} (2017), 79--94.

\bibitem{LiXiong18}
H.~Li and C.~Xiong.
\newblock{Stability of capillary hypersurfaces in a Euclidean ball.}
\newblock{Pac. J. Math.}, {\bf 297} (2018), 131--146.


\bibitem{Lopez06}
R. L\'opez.
\newblock{Spacelike hypersurfaces with free boundary in the Minkowski space under the effect of a timelike potential.}
\newblock{Comm. Math. Phys.}, {\bf 266} (2006), no.~2, 331--342.

\bibitem{Lopez08}
R. L\'opez.
\newblock{Stationary surfaces in Lorentz-Minkowski space.}
\newblock{Proc. Roy. Soc. Edinburgh Sect. A}, {\bf 138} (2008), no.~5, 1067--1096.

\bibitem{Montiel88}
S. Montiel.
\newblock{An Integral Inequality for Compact Spacelike Hypersurfaces in de Sitter Space and Applications to the Case of Constant Mean Curvature.}
\newblock{Indiana Univ. Math. J.}, {\bf 37} (1988), 909--917.

\bibitem{Oliker92}
V. Oliker.
\newblock{A Priori Estimates of the Principal Curvatures of Spacelike Hypersurfaces in de Sitter Space with Applications to Hypersurfaces in Hyperbolic Space.}
\newblock{American J. Math.}, {\bf 114} (1992), 605--626.

\bibitem{PyoSeo11}
J. Pyo and K. Seo.
\newblock{Spacelike capillary surfaces in the Lorentz-Minkowski space.}
\newblock{Bull. Aust. Math. Soc.}, {\bf 84} (2011), no.~3, 362--371.

\bibitem{RosVergasta95}
A.~Ros and E.~Vergasta.
\newblock{Stability for hypersurfaces of constant mean curvature with free boundary.}
\newblock{Geom. Dedic.}, {\bf 56} (1995), 19--33.

\bibitem{RosSouam97}
A.~Ros and R.~Souam.
\newblock{On stability of capillary surfaces in a ball.}
\newblock{Pac. J. Math.}, {\bf 178} (1997), 345--361.

\bibitem{Sato21} 
Y.~Sato. 
\newblock{Totally umbilical submanifolds in pseudo-Riemannian space forms.} 
\newblock{Tsukuba J. Math.}, {\bf 45} (2021), no.~2, 97--116.

\bibitem{T82}
A. Treibergs.
\newblock{Entire spacelike hypersurfaces of constant mean curvature in Minkowski space.}
\newblock{Invent. Math.}, {\bf 66} (1982), 39--56.

\bibitem{WangXia19}
G.~Wang and C.~Xia.
\newblock{Uniqueness of stable capillary hypersurfaces in a ball.}
\newblock{Math. Ann.}, {\bf 374} (2019), 1845--1882.

\end{thebibliography}
\end{document}